\documentclass[12pt]{amsart}

\usepackage{amsmath,amssymb,amscd,amsxtra,upref}
\usepackage{latexsym}
\usepackage[dvips]{graphics,epsfig}
\usepackage{enumerate}
\usepackage[colorlinks=true, pdfstartview=FitV, linkcolor=blue, citecolor=blue, urlcolor=blue]{hyperref}
\usepackage{color}
\usepackage{bbm}
\usepackage{mathabx}
\usepackage[utf8x]{inputenc}
\usepackage{pgfplots}
\pgfplotsset{compat=newest}
\usepackage[english]{babel}
\usepackage{tikz-cd}

\allowdisplaybreaks

 \makeatletter
\newtheorem*{rep@theorem}{\rep@title}
\newcommand{\newreptheorem}[2]{%
\newenvironment{rep#1}[1]{%
 \def\rep@title{#2 \ref*{##1}}%
 \begin{rep@theorem}}%
 {\end{rep@theorem}}}
\makeatother

\newtheorem{theorem}{Theorem}[section]
\newtheorem{lemma}[theorem]{Lemma}
\newtheorem{proposition}[theorem]{Proposition}
\newtheorem{corollary}[theorem]{Corollary}

\newtheorem{conjecture}[theorem]{Conjecture}
\newreptheorem{theorem}{Theorem}

\theoremstyle{remark}

\newtheorem{remark}{Remark}[section]

\newcommand{\R}{\mathbb{R}}
\newcommand{\Rn}{\mathbb{R}^n}
\newcommand{\Rnn}{\mathbb{R}^{2n}}
\newcommand{\Ha}{\mathcal{H}}
\newcommand{\Hn}{\mathbb{H}^n}
\newcommand{\He}{\mathbb{H}}
\newcommand{\C}{\mathbb{C}}

\newcommand{\Vt}{\mathbb{V}_{\theta}}
\newcommand{\V}{\mathbb{V}}
\newcommand{\Vp}{\mathbb{V}^{\perp}}
\newcommand{\U}{\mathbb{U}}
\newcommand{\Up}{\mathbb{U}^{\perp}}
\newcommand{\Vtp}{\mathbb{V}_{\theta}^{\perp}}
\newcommand{\Xt}{\mathbb{X}_{\theta}}
\newcommand{\X}{\mathbb{X}}

\newcommand{\Gt}{\mathbb{G}_{\theta}}
\newcommand{\G}{\mathbb{G}}
\newcommand{\Phit}{\Phi_{\theta}}
\newcommand{\Gh}{G_h(n,m)}
\newcommand{\Ghh}{G_h(n,1)}
\newcommand{\munm}{\mu_{n,m}}
\newcommand{\PV}{P_{\mathbb{V}}}
\newcommand{\PVp}{P_{\mathbb{V^{\perp}}}}
\newcommand{\VH}{\mathbb{V}\backslash\Hn}
\newcommand{\dvh}{d_{\V\backslash\Hn}}
\newcommand{\VoH}{\V_0\backslash\Hn}
\newcommand{\Vpo}{\V_0^{\perp}}

\DeclareMathOperator{\supp}{supp}
\DeclareMathOperator{\spn}{span}
\DeclareMathOperator{\Int}{Int}

\begin{document}

\title[Right coset projections on the Heisenberg group]
{Dimension distortion by right coset projections in the Heisenberg group}

\author[T.~L.~J.~Harris]{Terence L.~J.~Harris}
\address{Department of Mathematics, University of Illinois, Urbana, IL 61801, U.S.A.}
\email{terence2@illinois.edu}

\author[C.~N.~Y.~Huynh]{Chi N.~Y.~Huynh}
\address{Department of Mathematics, University of Illinois, Urbana, IL 61801, U.S.A.}
\email{nyhuynh2@illinois.edu}

\author[F.~Rom\'{a}n-Garc\'{i}a]{Fernando Rom\'{a}n-Garc\'{i}a}
\address{Department of Mathematics, University of Illinois, Urbana, IL 61801, U.S.A.}
\email{romanga2@illinois.edu}



\thanks{We thank Jeremy Tyson for suggesting the problem to us.}

\begin{abstract} We study the family of vertical projections whose fibers are right cosets of horizontal planes in the Heisenberg group, $\Hn$. We prove lower bounds for Hausdorff dimension distortion of sets under these mappings 
with respect to the natural quotient metric, which we show behaves like the Euclidean metric in this context. Our bounds are sharp in a large part of the dimension range, and we give conjectural sharp lower bounds for the remaining range. Our approach also lets us improve the known almost sure lower bound for the standard family of vertical projections in $\Hn$ for $n \geq 2$.
\end{abstract}

\maketitle

\section{Introduction}
 \begin{sloppypar}

The study of dimension distortion by projections dates back to J.~Marstrand's 1954 paper \cite{Marstrand}. Among many other things, it was shown that for an analytic 
set $A\subset \R^2$, $\dim P_{\theta}(A)=\min\{\dim A, 1\}$ for $\Ha^1$-almost all $\theta\in[0,\pi)$, where $P_{\theta}:\R^2\to \ell_{\theta}$ is the orthogonal projection onto the line with terminal angle $\theta$. Moreover, it was shown that if $\dim A>1$ then $\Ha^1(P_{\theta}(A))>0$ for $\Ha^1$-almost all $\theta\in[0,\pi)$. Over time, this result has been expanded and generalized in many directions. For instance in \cite{Kaufman}, R. Kaufman introduced a potential theoretic approach that streamlined Marstrand's proof, and  using this approach P.~Mattila generalized the result to higher dimensions \cite{MattilaProj}. The general result, including the Besicovitch-Federer characterization of unrectifiability (\cite{Bes1}, \cite{Federer}), is stated in the following theorem.

\begin{theorem}\label{EucProj}
Let $A\subset\Rn$ be an analytic set of dimension $s$. 
\begin{enumerate}
\item If $s\leq m$, $\dim P_V(A) = s$ for almost every $m$-dimensional subspace $V$.
\item If $s> m$, $\Ha^m( P_V(A)) > 0$  for almost every $m$-dimensional subspace $V$.
\item If $s> 2m$, $\Int( P_V(A)) \neq\varnothing$  for almost every $m$-dimensional subspace $V$.\end{enumerate} 
Moreover, in the case where $s=m$ and with the added hypothesis that $\Ha^m(A)<\infty$, $A$ is purely $m$-unrectifiable if and only if $\Ha^m( P_V(A)) = 0$ for almost every $m$-dimensional subspace $V$. 
\end{theorem}

Analogous, but in some cases weaker, results have been obtained when projections are restricted to a subfamily of planes \cite{Hproj2, IsoProj, FasslerOrponen, OberOber, Chen, OrpVen, THarris1}. In \cite{PS} the authors introduced the concept of transversal families of maps thus giving a vast generalization of Theorem~\ref{EucProj} which extended the result to many more families of mappings. The problem has also been studied outside of the Euclidean setting, specifically in the Heisenberg group, in \cite{HProj1, Hproj2}. There, the story is far from over. Two distinct families of ``projections" arise naturally in this context, known as homogeneous projections. Dimension distortion by one of these families, that of horizontal projections, can be tackled using transversality, but the other family, that of vertical projections, is not transversal in the sense of Peres and Schlag and is otherwise quite difficult to work with. Improving the known dimension distortion bounds in this context continues to be an active area of research with improvements being made recently in \cite{THarris2}. In this paper we continue the work in this direction by studying another natural, yet unstudied, family of projections in the Heisenberg group. Our approach also improves the known dimension distortion bound for the standard family of homogeneous projections studied in \cite{Hproj2}.

The $n$th Heisenberg group is defined as the manifold $\Hn:=\R^{2n}\times\R$ with typical point denoted by $(z,t)=(x_1,\ldots,x_n,y_1\ldots,y_n,t)$ where for $j=1,\dots,n$, $z_j=x_j+iy_j$.  As such, we will identify $\mathbb{C}^n$ with $\R^{2n}$ (and e.g.~write $iz$ for pointwise scalar multiplication of $z \in \mathbb{C}^n$ by $i$). We endow this manifold with the group law $(z,t)*(w,s)=(z+w,t+s+\frac{1}{2}\omega(z,w))$, where $w=(u_1,\ldots,u_n,v_1,\ldots,v_n)$, and $\omega(z,w)=\sum_{j=1}^n(x_jv_j-y_ju_j)$. This group law makes $\Hn$ a Lie group with left invariant vector fields \[X_j=\frac{\partial}{\partial x_j}-\frac{y_j}{2}\frac{\partial}{\partial t},\ Y_j=\frac{\partial}{\partial y_j}+\frac{x_j}{2}\frac{\partial}{\partial t},\ T=\frac{\partial}{\partial t}  \text{ for } j=1,\ldots,n. \]

 For any given $j$,  $[X_j,Y_j]=T$, so $\mathcal{H}=\spn\{X_j,Y_j:\ j=1,\ldots,n\}$ forms a bracket generating distribution. We say an absolutely continuous curve $\gamma:[0,1]\to\Hn$ is horizontal if \[\dot{\gamma}(s)\in\mathcal{H}_{\gamma(s)}\ \text{for a.e.~}s\in[0,1].\] By declaring $\{X_j,Y_j:\ j=1,\ldots, n\}$ to be orthonormal, we can compute the (horizontal) length of $\gamma$ in the usual way.  We will denote the length of $\gamma$ by $|\gamma|$. The bracket generating condition enables the definition of a Carnot-Carathéodory distance in all of $\Hn$ via
\[d_{cc}(p,q)=\inf\{|\gamma|:\ \gamma\text{ is horizontal, and } \gamma(0)=p,\ \gamma(1)=q\}.\] 
The Korányi gauge $\left\lVert (z,t)\right\rVert_{\mathbb{H}^n}^4=|z|^4+16t^2$ also gives a left invariant metric (known as the Korányi metric) given by $d_{\mathbb{H}^n}(p,q)=\left\lVert q^{-1}*p\right\rVert$. These two metrics are bi-Lipschitz equivalent.

For $r>0$ the non-isotropic dilations $\delta_r(z,t)=(rz,r^2t)$ give $\Hn$ a homogeneous structure. This enables the definition of homogeneous subgroups as subgroups which are closed under dilations. These subgroups come in two kinds, those contained in $\C^n\times\{0\}$ (horizontal), and those containing the entire $t$-axis (vertical). The $t$-axis is a homogeneous subgroup, one without a complementary horizontal subgroup. The horizontal subgroups $V \times \{0\}$ coincide with isotropic subspaces $V$ of $\C^n$, and their (Euclidean) orthogonal complements $V^{\perp}\times\R$ are vertical subgroups (here an isotropic subspace means one on which the symplectic form $\omega$ vanishes identically). We denote the Grassmannian of isotropic $m$-planes in $\R^{2n}$ as $\Gh$, and for $V\in\Gh$, we denote the corresponding horizontal and vertical subgroups by $\V$ and $\Vp$ respectively. For each $V\in\Gh$, $\Vp$ is a normal subgroup of $\Hn$, and we have a semi-direct splitting $\Hn=\V\ltimes\Vp$. Since the group $\Vp$ is normal, the splitting can also be taken to be $\Hn=\Vp\rtimes\V$. These splittings induce projection maps $\PV$ onto the horizontal subgroup $\V$, and $\PVp^R$, $\PVp^L$ onto the vertical subgroup $\Vp$. Here $\PVp^R$ is induced by the first mentioned splitting, and its fibers are right cosets of the subgroup $\V$. In the same way, $\PVp^L$ is induced by the second splitting and its fibers are left cosets of the horizontal subgroup $\V$.  Turns out, $\PV$ agrees with the Euclidean orthogonal projection onto the subspace $V$, while $\PVp^R$, and $\PVp^L$ can be defined via the group law by $\PVp^R(p)=\PV(p)^{-1}p,\ \PVp^L(p)=p\PV(P)^{-1}$. Since the group law is non-commutative, these two maps are inherently different, although they are related by the equation $\PVp^L(p)=-\PVp^R(-p)$. It is important to note that given a set $A\subset\Hn$ $\dim_{\Hn}A\neq\dim_{\Hn}(-A)$ in general. It is therefore expected that these maps behave differently when it comes to dimension distortion.

The group $U(n)$ of complex unitary matrices, which may be identified as a subgroup of $O(2n)$, preserves the symplectic form $\omega$ (see \cite[Chapter 3]{Mattila2}). This group acts smoothly and transitively on $\Gh$, and each $R\in U(n)$ induces an isometry of $\Hn$ given by $\mathcal{R}(z,t)=(Rz,t)$. Therefore, for any two horizontal subgroups $\V$ and $\V'$ there is an $R_0\in U(n)$ such that $\V=\mathcal{R}_0\V'$. Since $U(n)$ has a unique probability Haar measure, the space $\Gh$ inherits a unique $U(n)$-invariant probability measure, which we denote by $\munm$. This in turn allows us to put a measure on the set of horizontal (resp.~vertical) subgroups of Hausdorff dimension $m$ (resp.~$2n+2-m$) in $\Hn$, specifically, one simply uses the measure $\munm$ by appealing to the aforementioned correspondence between horizontal (resp.~vertical) subgroups and $\Gh$.

The vertical projections $\PVp^L$, together with horizontal projections, have been heavily studied 
in the context of Hausdorff dimension distortion (\cite{HProj1}, \cite{Hproj2}, \cite{FasslerHovila}, \cite{THarris2}). These projections also play a pivotal role in the theory of rectifiable sets in $\Hn$ (\cite{RectifiabilityinH}). Here we intend to initiate the study of the projection $\PVp^R$ in the context of dimension distortion. Whereas the fibers of the map $\PVp^L$ are horizontal lines, the fibers of $\PVp^R$ are not horizontal. It is therefore not very natural to consider $\PVp^R$ as a map from $(\Hn,d_{cc})$ to $(\Vp,d_{cc}\lfloor_{\Vp})$. In $\He^1$, the maps $\PVp^R$ have already been studied in other contexts (see for instance \cite{HtoG}) where a natural metric arises on the image of  $\PVp^R$. We study dimension distortion in the context of this, ``more natural", metric by first generalizing it to higher dimensions. Our main result is as follows.

\begin{theorem}\label{Hnbounda}  For $1 \leq m \leq n$ and any Borel set $A \subseteq \mathbb{H}^n$,  
\begin{equation} \label{euclidean2a}
 \dim_E P^L_{\mathbb{V}^{\perp}} (A), \ \dim_E P^R_{\mathbb{V}^{\perp}} (A) \geq \begin{cases}  \dim_EA & \text{ if }\dim_EA\in[0,2n-m]  \\
2n-m & \text{ if } \dim_EA\in[2n-m,2n]\\
\dim_EA-m & \text{ if } \dim_EA\in[2n,2n+1]
\end{cases}
\end{equation}
for $\mu_{n,m}$-a.e.~$V \in G_h(n,m)$, and
\begin{equation} \label{koranyi2a}
 \dim_{\mathbb{V} \backslash \mathbb{H}^n} P^R_{\mathbb{V}^{\perp}} (A)
\geq \begin{cases} (\dim_{\Hn}A)/2 &\text{ if }\dim_{\Hn}A\in[0,2]\\
\dim_{\Hn}A-1 &\text{ if }\dim_{\Hn}A\in[2,2n-m+1]\\
2n-m  &\text{ if }\dim_{\Hn}A\in[2n-m+1,2n+1]\\
\dim_{\Hn}A-m-1 &\text{ if }\dim_{\Hn}A\in[2n+1,2n+2]\\
  \end{cases} \end{equation}
for $\mu_{n,m}$-a.e.~$V \in G_h(n,m)$. If $\dim_E A \leq 2n-m$ then \eqref{euclidean2a} is sharp, and if $\dim_{\mathbb{H}^n} A \leq 2n+1-m$ then \eqref{koranyi2a} is sharp. 
\end{theorem}
Here, $\dvh$ refers to this aforementioned ``more natural" metric on $\Vp$ while $\dim_E$ and $\dim_{\Hn}$ refer to the Hausdorff dimension with respect to the Euclidean and Heisenberg metrics, respectively (see \cite{SubGeo}). Our main idea is to obtain a projection theorem in the Heisenberg group by first considering the Euclidean metric on both sides and then applying some kind of ``dimension comparison principle''. This is natural for right coset projections because the resulting bound obtained is sometimes sharp. 
We remark that the Euclidean-Euclidean dimension distortion problem for vertical projections in $\mathbb{H}$ seems to have been first posed in \cite[p.~296]{MR3276006}. At least one instance of applying Euclidean methods and dimension comparison to projection bounds in the Heisenberg group can be found in the proof of Proposition 4.9 in \cite{Hproj2}.

By dimension comparison, Theorem~\ref{Hnbounda} leads to the following almost sure dimension bound for the standard (left-coset) projection problem.
\begin{theorem}\label{LeftCoset}
For $1 \leq m \leq n$ and any Borel set $A \subseteq \mathbb{H}^n$,  
\begin{equation} \label{leftcosetkoranyi}
\dim_{\Hn} P^L_{\mathbb{V}^{\perp}} (A) \\
\geq\begin{cases}
\dim_{\Hn}A-1 &\text{ if }\dim_{\Hn}A\in[2,2n-m+1]\\
2n-m  &\text{ if }\dim_{\Hn}A\in[2n-m+1,2n+1]\\
  \end{cases}  \end{equation} 
for $\mu_{n,m}$-a.e.~$V \in G_h(n,m)$.
\end{theorem}
 Previously, the best known almost sure lower bound for this problem (in $\Hn$ with $n>1$ and $\dim_{\mathbb{H}^n}A \leq m+2$) was
\[\dim_{\Hn}P_{\Vp}^L(A)\geq \min\{\dim_{\Hn}A,1\}\ \text{for $\munm$-almost all } V\in\Gh.\]
This bound also holds when $n=1$, though there it is not the best known. 
The best known universal lower bound was
\begin{multline*}
\dim_{\Hn}P_{\Vp}^L(A)\geq\\
\max\left\{0,\frac{\dim_{\Hn}A-m}{2},\dim_{\Hn}A-m-1,2(\dim_{\Hn}A-n-1)-m\right\}.
\end{multline*}
From this, the best possible almost sure lower bound was
\begin{multline*}
\dim_{\Hn}P_{\Vp}^L(A)\geq\\
\max\left\{\min\{\dim_{\Hn}A,1\},\frac{\dim_{\Hn}A-m}{2},\dim_{\Hn}A-m-1,2(\dim_{\Hn}A-n-1)-m\right\},
\end{multline*}
for $\munm$-almost every $V\in\Gh$.
Therefore, Theorem~\ref{LeftCoset} improves this almost sure lower bound in the range $\dim_{\Hn}A\in[2,2n+1]$. The new lower bound reads
\[\dim_{\Hn}P_{\Vp}^{L}(A)\geq\begin{cases}
\dim_{\Hn}A & \text{ if } \dim_{\Hn}A\in[0,1]\\
1 & \text{ if } \dim_{\Hn}A\in[1,2]\\
\dim_{\Hn}A-1 & \text{ if } \dim_{\Hn}A\in[2,2n-m+1]\\
2n-m & \text{ if } \dim_{\Hn}A\in[2n-m+1,2n+1]\\
2(\dim_{\Hn}A-n-1)-m & \text{ if } \dim_{\Hn}A\in[2n+1,2n+2],
\end{cases}\]
for $\munm$-almost every $V\in\Gh$. 

For $n > 1$, we do not know if the lower bounds in Theorem~\ref{Hnbounda} are sharp for $\dim_EA\geq 2n-m$ and $\dim_{\Hn}A\geq 2n+1-m$, but we suspect the answer is no. For $\dim_{\mathbb{H}^n} A > 2$ we predict the lower bound $\dim_{\Hn}A-1$ to hold up to $\dim_{\Hn}A=2n+2-m$; the example in the proof of Theorem~\ref{Hnbounda} shows this would be sharp. The conjectured lower bounds are given below; in all cases it is assumed that $1 \leq m \leq n$. 

\begin{conjecture}[{see \cite[Conjecture~1.5]{HProj1} for the case $n=1$}] \label{conjA} Let $A \subseteq \mathbb{H}^n$ be a Borel set. If $\dim_{\mathbb{H}^n} A \leq 2n+2-m$ then 
\[ \dim_{\mathbb{H}^n} P_{\mathbb{V}^{\perp}}^L (A) \geq  \dim_{\mathbb{H}^n} A \quad \text{ for a.e.~$V \in G_h(n,m)$,} \]
and if $\dim_{\mathbb{H}^n} A > 2n+2-m$ then 
\[ \mathcal{H}^{2n+2-m}_{d_{cc} }\left( P_{\mathbb{V}^{\perp}}^L (A) \right) >0  \quad \text{ for a.e.~$V \in G_h(n,m)$.} \]
\end{conjecture} 

\begin{conjecture} \label{conjD}  For any Borel set $A \subseteq \mathbb{H}^n$,
\[ \dim_{\mathbb{V} \backslash \mathbb{H}^n} P_{\mathbb{V}^{\perp}}^R (A) \geq \min\left\{ \dim_E A, 2n-m+1 \right\} \quad \text{ for a.e.~$V \in G_h(n,m)$.} \]
\end{conjecture} 

\begin{conjecture}  \label{conjE} For any Borel set $A \subseteq \mathbb{H}^n$,
\[ \dim_E P_{\mathbb{V}^{\perp}}^{L} (A) , \ \dim_E P_{\mathbb{V}^{\perp}}^{R} (A) \geq \min\left\{ \dim_E A, 2n-m+1 \right\} \quad \text{ for a.e.~$V \in G_h(n,m)$.} \]
\end{conjecture} 

\begin{conjecture} \label{conjB} For any Borel set $A \subseteq \mathbb{H}^n$,
\[ \dim_{\mathbb{V} \backslash \mathbb{H}^n} P_{\mathbb{V}^{\perp}}^R (A) \geq \min\left\{ \max\left\{ \frac{\dim_{\mathbb{H}^n} A}{2}, \dim_{\mathbb{H}^n} A -1 \right\}, 2n-m+1 \right\}, \]
for a.e.~$V \in G_h(n,m)$.
\end{conjecture} 

\begin{conjecture} \label{conjC}  For any Borel set $A \subseteq \mathbb{H}^n$,
\[  \dim_E P_{\mathbb{V}^{\perp}}^{L} (A) 
\geq \min\left\{ \max\left\{ \frac{\dim_{\mathbb{H}^n} A}{2}, \dim_{\mathbb{H}^n} A -1 \right\}, 2n-m+1 \right\}, \]
for a.e.~$V \in G_h(n,m)$.
\end{conjecture} 

\begin{conjecture} \label{conjF}  For any Borel set $A \subseteq \mathbb{H}^n$,
\[  
 \dim_E P_{\mathbb{V}^{\perp}}^{R} (A)  
\geq \min\left\{ \max\left\{ \frac{\dim_{\mathbb{H}^n} A}{2}, \dim_{\mathbb{H}^n} A -1 \right\}, 2n-m+1 \right\}, \]
for a.e.~$V \in G_h(n,m)$.
\end{conjecture} 

All these conjectures are sharp if true; the connections between them are pictured below.  The relations and sharpness will be shown at the end of Section~\ref{dimensiondistortion}.
\begin{equation} \label{conjecturerelations}
\begin{tikzcd}
& \text{Conj.~\ref{conjE}} \arrow[Leftrightarrow]{r}{} \arrow[Rightarrow]{d}{}    & \text{Conj.~\ref{conjD}} \arrow[Rightarrow]{d}{}  \\ 
\text{Conj.~\ref{conjA}} \arrow[Rightarrow]{r}{}  & \text{Conj.~\ref{conjC}} 
  & \text{Conj.~\ref{conjB}} \arrow[Leftrightarrow]{r}{} & \text{Conj.~\ref{conjF}}
\end{tikzcd} \end{equation}


Here we also include graphs summarizing our results on a.e.~Heisenberg and Euclidean dimension distortion.



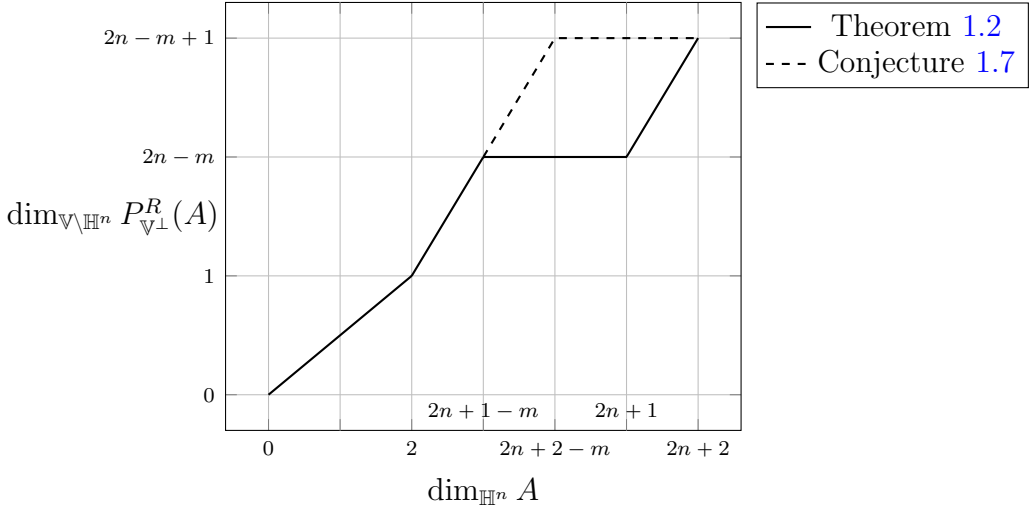
\begin{center}
\begin{tikzpicture}[
declare function={
    func(\x)= (\x<=2) * (0.5*\x)   +
              and(\x>2, \x<=3) * (\x-1)     +
              and(\x>3, x<=5) * (2)				+
              and(\x>5, x<=6) * (\x-3);
		func2(\x)= (\x<=2) * (0.5*\x)   +
              and(\x>2, \x<=4) * (\x-1)     +
              and(\x>4, x<=6) * (3);			
  }]
  \begin{axis}[samples=300,
		grid=major,
		xtick={0,2,4,6},
    xticklabels={{\tiny$0$},{\tiny$2$},{\tiny $2n+2-m$},{\tiny$2n+2$}},
		extra x ticks={1,3,5},
    extra x tick labels={,{\tiny$2n+1-m$},{\tiny$2n+1$}},
    extra x tick style={
                xticklabel style={anchor=south}
            },     
		ytick={0,1,2,3},
    yticklabels={{\tiny$0$},{\tiny$1$},{\tiny$2n-m$},{\tiny$2n-m+1$}},
    xlabel=$\dim_{\mathbb{H}^n} A$,
    ylabel={$\dim_{\mathbb{V} \backslash \mathbb{H}^n} P^R_{\mathbb{V}^{\perp}} (A)$},
		ylabel style={yshift=-10ex,rotate=-90},
		legend style={
legend pos=outer north east,
}
  ] 
 \addplot[black,thick,domain=0:6]{func(x)}; 
 \addplot[dashed,thick, domain=3:6]{func2(x)};
\legend{Theorem~\ref{Hnbounda},Conjecture~\ref{conjB}}
  \end{axis}
\end{tikzpicture}
\end{center}

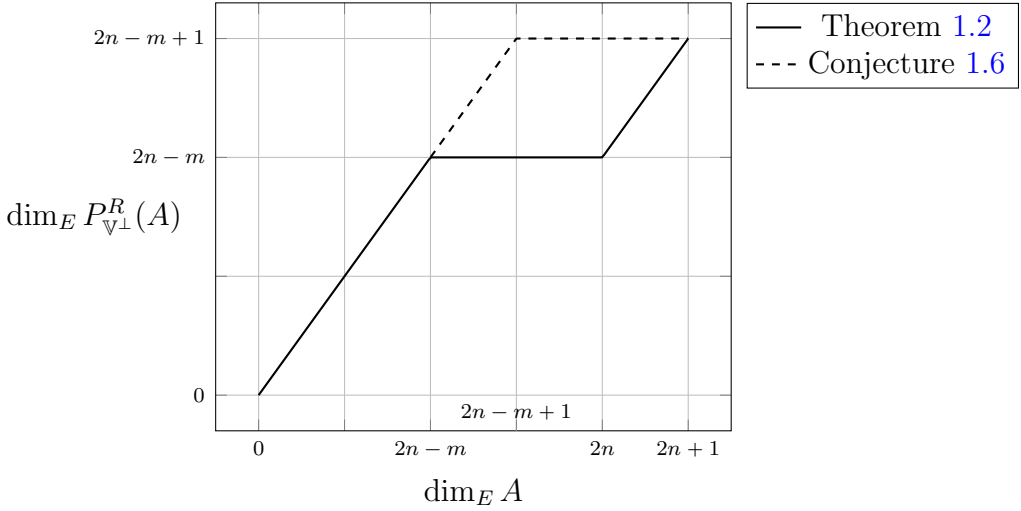
\begin{center}
\begin{tikzpicture}[
declare function={
    func(\x)= (\x<=2) * (\x)   +
              and(\x>2, \x<=4) * (2)     +
              and(\x>4, x<=5) * (\x-2);
		func2(\x)= (\x<=3) * (\x)   +
              and(\x>3, \x<=5) * (3);					
  }]
  \begin{axis}[samples=200,
		grid=major,
		xtick={0,1,2,4,5},
    xticklabels={{\tiny$0$},,{\tiny$2n-m$},{\tiny$2n$},{\tiny$2n+1$}},
		extra x ticks={3},
    extra x tick labels={{\tiny$2n-m+1$}},
    extra x tick style={
                xticklabel style={anchor=south}
            },     
		ytick={0,1,2,3},
    yticklabels={{\tiny$0$},,{\tiny$2n-m$},{\tiny$2n-m+1$}},
    xlabel=$\dim_E A$,
    ylabel={$\dim_E P^R_{\mathbb{V}^{\perp}} (A)$},
		ylabel style={yshift=-8ex,rotate=-90},
		legend style={
legend pos=outer north east,
}
  ] 
 \addplot[black,thick,domain=0:5]{func(x)}; 
 \addplot[dashed,thick, domain=2:5]{func2(x)};
\legend{Theorem~\ref{Hnbounda},Conjecture~\ref{conjE}}
  \end{axis}
\end{tikzpicture}
\end{center}

Finally, in the first Heisenberg group $\mathbb{H}$ there is a small improvement possible to Theorem~\ref{Hnbounda}, which we show in Section~\ref{first}. With Euclidean metrics on each side, Corollary~\ref{1dbound} is a better a.e.~lower bound than Theorem~\ref{Hnbounda} for $\dim_E A \in \left(1, 5/2\right)$.

\section{Right coset projections in \texorpdfstring{$\Hn$}{the first Heisenberg group}}\label{H1Grushin}
In this section we will first introduce the Grushin plane, which will come back later in connection with right coset quotient spaces. Then we will describe the right coset quotient space by vertical subgroups together with the corresponding vertical projections. Finally we will restrict to the case of vertical subgroups of co-dimension one where we have a clear description of the metric structure of the space and the aforementioned connection with the Grushin plane arises. It is worth mentioning that the connection between the Heisenberg group and the Grushin plane has been studied before (see for instance \cite{HtoG}, \cite{RotPreissStein} and \cite[(3) p.293]{Folland}). 

In this section we only consider the projections $\PVp^R$ which we will simply denote by $\PVp$. In addition, for $1\leq m\leq n$, the notation $\He^{n-m}$ will be frequently used. It is therefore important to emphasize that this notation signifies the $(n-m)$th Heisenberg group, $\C^{n-m}\times\R$, with all of its structure. In particular, when $n=m$, $\He^{n-m}$ is simply the ``$t$-axis", $\C^0\times\R=\R$, with standard addition and metric $d_{\He^{n-m}}=2 d_E^{1/2}$.

\subsection{The Grushin Plane}
The Grushin plane is the manifold $\mathbb{G}=\R^2$ with vector fields
\begin{equation} \label{vectorfields} \begin{cases}T =-v\frac{\partial}{\partial \tau}\\
V =\frac{\partial}{\partial v},
\end{cases} \end{equation}
where $(v,\tau) \in \mathbb{R}^2$. These vector fields span the whole tangent space at every point outside of the singular set $\{v=0\}$, and by taking them to be orthonormal there, we get a line form
\[ds^2=dv^2+\frac{d\tau^2}{v^2}\]
on $\R^2 \setminus \{(0,\tau ):\tau \in \R\}$. One can check that $[T,V]=\frac{\partial}{\partial \tau}$, which allows us to extend this metric to a Carnot-Carathéodory path distance in all of $\R^2$.  The resulting metric, denoted by $d_{\mathbb{G}}$, turns $\G$ into a non-equiregular sub-Riemannian manifold whose horizontal curves are curves that have horizontal tangent at every point of intersection with the critical line. That is to say, $\gamma:[0,1]\to\G$ is horizontal if there exist integrable functions $a$ and $b$ such that
 \[\dot{\gamma}(s)=a(s)T+b(s)V, \] 
for a.e.~$s \in [0,1]$. The length of $\gamma$ is then given by
\[\int_0^1 \left[a(s)^2+b(s)^2\right]^{1/2} \, ds.\]
If we write $\gamma(s)=(v(s),\tau(s))$, a more explicit formula for the length is 
\begin{equation}\label{Glength}
\Lambda_{\G}=\int_0^1\left[\dot{v}(s)^2+\frac{\dot{\tau}(s)^2}{v(s)^2}\right]^{1/2} \, ds.
\end{equation}

For each $t_0$ the vertical translation map $(v,t)\to(v,t+t_0)$ is an isometry of $\G$. This can also be seen as a non-transitive group action by $\R$ whose orbits are vertical lines, in particular, the orbit of $0$ is the critical line $v=0$. One interesting property of the Grushin metric, that will come back later in the discussion, is that the restriction of the distance to the critical line is comparable to the square root of the Euclidean distance. Therefore, this ``copy" of $\R$ is embedded into $\G$ in a ``snowflaked" way. In contrast, the restriction of the distance to any other vertical line is Riemannian. 
\end{sloppypar} 
\subsection{The right coset quotient space}\label{rightcoset}
For $1\leq m\leq n$, given $V\in\Gh$ we consider the quotient space of right cosets of $\V$ in $\Hn$,
\[\V\backslash\Hn:= \left\{ \mathbb{V}p :p\in\Hn \right\},\]
endowed with the quotient distance
\[\dvh(\V p,\V p')=\inf\left\{d_{cc}(qp,p'):\ q\in\V\right\}.\] 
There is a unique way to write elements of $\VH$ as $\V q$ with $q\in\Vp$. Therefore $\VH$ is identified with $\Vp$  by the map $\V q\mapsto q$. This map coincides with the map on $\VH$ induced by $\PVp$, that is $\PVp(\mathbb{V}p)=\{\PVp(p)\}$. 
\begin{lemma}
For each fixed $\V$, the map 
\[\PVp:(\Hn,d_{cc})\to \left(\Vp,\dvh\right)\]
is $1$-Lipschitz. 
\end{lemma}
\begin{proof}
Indeed, if $p,p'\in\Hn$ we have
\[\dvh(\PVp(p),\PVp(p'))=\inf_{q\in\V}d_{cc}(q\PVp(p),\PVp(p')).\] An upper bound is found by choosing a specific $q\in\V$. In particular, choosing $q=P_{\V}(p')^{-1}P_{\V}(p)$, and appealing to the left invariance of $d_{cc}$ we see that,
\[\dvh(\PVp(p),\PVp(p'))\leq d_{cc}(p,p'). \qedhere \]
\end{proof}
Denoting by $\pi_W$ the Euclidean orthogonal projection onto $W$, an explicit formula for the projection is given by
\begin{equation}\label{projetionformula}
\PVp(z,t)=\left(\pi_{V^{\perp}}(z),t-\frac{1}{2}\omega(\pi_V(z),\pi_{V^{\perp}}(z))\right)
\end{equation}

The space $\VH$ inherits a rich structure from $\Hn$ which allow us to have a more intuitive understanding of the space.

The unitary group, $U(n)$, acts smoothly and transitively on $\Gh$ and isometrically on $\Hn$ via $(z,t)\to(Rz,t),\  (R\in U(n))$, therefore understanding the metric properties of $\VoH$ for a fixed $\V_0$ will get us the same properties for $\VH$ in general. Hence, to simplify computations, fix the horizontal subgroup 
\[ \mathbb{V} = \V_0:=\{(x_1,\ldots,x_m,0,\ldots, 0): x_j\in\R\},\]
for the rest of this section. This gives us
\[\Vp=\Vpo=\{(0,\ldots,0,x_{m+1},\ldots, x_n, y_1,\ldots, y_n,t):\ x_j, y_j, t\in\R \}.\]
With this concrete setting, we discuss some of the symmetries of the space $\VH$.

\subsection*{Homogeneous dilations}

The space $\VH$ admits homogeneous dilations. Although these dilations are defined on $\VH$ we abuse notation using the same symbol as for the Heisenberg dilations since the dilations on $\VH$ are nothing more than the dilations on $\Hn$ that factor through the quotient map. For each $r>0$ the map $\delta_r:\VH\to\VH$ given by 
\[\delta_r(0,\ldots,0,x_{m+1}, \ldots, y_n,t)= \left(0,\ldots,0, r x_{m+1}, \ldots, ry_n,r^2t\right),\]
is homogeneous of degree 1 with respect to $\dvh$. 
Indeed:
\[\dvh(\delta_r(p),\delta_r(p'))=\inf_{q\in\V}d_{cc}(q\delta_r(p),\delta_r(p'))=r\inf_{q\in\V}d_{cc}(\delta_{1/r}(q)p,p')=r\dvh(p,p').\]
The last equality follows from the fact that $\V$ is homogeneous (so that $\delta_{1/r}(q)\in\V$).

\subsection*{Group action by \texorpdfstring{$\He^{n-m}$}{lower dimensional Heisenberg group}}

We embed $\He^{n-m}$ in $\Hn$ by the map $\xi\mapsto \widehat{\xi}$ given by,
\[(u_1,\ldots, u_{n-m},v_1,\ldots, v_{n-m},\tau)\mapsto (0,\ldots,0,u_1,\ldots ,u_{n-m},0,\ldots,0, v_1,\ldots, v_{n-m},\tau),\]
where in the right hand side the first $m$ coordinates and coordinates $n+1$ through $n+m$ are all zero. With this notation we can see that $\He^{n-m}$ acts on $\Hn$ by ``left translation" via the map\[L_{\xi}p=\widehat{\xi}p.\]
To see that this action is isometric, note that for each $\xi\in\He^{n-m}$, $\widehat{\xi}$ commutes with elements of $\V$.  Indeed, writing $q=(z,0)\in\V$ and $\widehat{\xi}=(\widehat{w},\tau)$, it is not hard to see that $\omega\left(\widehat{w},z\right)=0$. Because of this,
\begin{align*}
\dvh(L_{\xi}p,L_{\xi}p')&=\inf_{q\in\V}d_{cc}\left(q\widehat{\xi}p,\widehat{\xi}p'\right)\\
&=\inf_{q\in\V}d_{cc}\left(\widehat{\xi}qp,\widehat{\xi}p'\right)\\
&=\inf_{q\in\V}d_{cc}(qp,p')=\dvh(p,p').
\end{align*}
This action is smooth with respect to the quotient topology but it is not transitive. For a point $(0,\ldots,0,x_{m+1},\ldots,x_n,y_1,\ldots,y_n,t)\in\Vp$ its orbit consists exactly of all other points of the form $(0,\ldots,0,x_{m+1}',\ldots,x_n',y_1,\dotsc ,y_m, y_{m+1}',y_n',t')$. Therefore, the orbit space is parametrized by $\R^{m}$.

\subsection*{Group action by \texorpdfstring{$U(n-m)$}{lower dimensional unitary group}}

Similarly, we embed $U(n-m)$ into $U(n)$  via the map $R\mapsto \widetilde{R}$ given for each $z=(x_1,\ldots, x_n,y_1,\ldots, y_n)\in\Rnn$ by
\[\widetilde{R}z=\widetilde{z}.\]
Here \[\widetilde{z}=(x_1\ldots, x_m, \widetilde{x}_{m+1},\ldots,\widetilde{x}_n, y_1, \dotsc, y_{m},\widetilde{y}_{m+1},\ldots,\widetilde{y}_{n})\] with \[(\widetilde{x}_{m+1},\ldots,\widetilde{x}_n,\widetilde{y}_{m+1},\ldots, \widetilde{y}_n)=R(x_{m+1},\ldots, x_{n},y_{m+1},\ldots, y_{n}).\]
In this way $U(n-m)$ acts on $\VH$ via $p\mapsto\widehat{\mathcal{R}}p := \left(\widetilde{R}z,t\right)$ where $p=(z,t)\in\Vp\simeq\VH$. Once again, it is not hard to check that this action, as an action naturally extended to all of $\Hn$, fixes $\V$ pointwise. Therefore $\widehat{\mathcal{R}}(qp)=q\widehat{\mathcal{R}}p$ for each $q\in\V$ and $p\in\Vp$. Since $U(n)$ acts isometrically on $\Hn$, it follows that 
\begin{align*}
\dvh\left(\widehat{\mathcal{R}}p,\widehat{\mathcal{R}}p'\right)&=\inf_{q\in\V}d_{cc}\left(q\widehat{\mathcal{R}}p,\widehat{\mathcal{R}}p'\right)\\
&=\inf_{q\in\V}d_{cc}\left(\widehat{\mathcal{R}}(qp),\widehat{\mathcal{R}}p'\right)\\
&=\inf_{q\in\V}d_{cc}(qp,p')=\dvh(p,p').
\end{align*}
Like the $\He^{n-m}$ action, the action by $U(n-m)$ is smooth but not transitive. The orbit of a point 
$(0,\ldots,0,x_{m+1},\ldots,x_n,y_1,\ldots,y_n,t)\in\Vp$ consists of all other points of the form  $(0,\ldots,0,x_{m+1}',\ldots,x_n',y_1,\ldots y_m, y_{m+1}',y_n',t)$. Therefore, the orbit space is parametrized by $\R^{m+1}$.

The group action by $\He^{n-m}$ reveals that there are ``$\R^m$ many" copies of the set $\He^{n-m}$ embedded in $\Vp$ in a natural way. More precisely, using the notation $p=(x_1,x_2,y_1,y_2,t)\in\R^{m}\times\R^{n-m}\times\R^m\times\R^{n-m}\times\R=\He^n$, for a fixed $\widetilde{y}\in\R^m$ we denote by $U_{\widetilde{y}}$ the orbit $U_{\widetilde{y}}=\{L_{\xi}(0,0, \widetilde{y},0,0)\in\Hn:\xi\in\He^{n-m}\}$. The map $\He^{n-m}\to U_{\widetilde{y}}$ given by $(x,y,t)\to(0,x,\widetilde{y},y,t)$ gives a natural embedding of the set $\He^{n-m}$ into $\Vp$.

\begin{proposition}\label{restriction}
The restrictions of $d_{\VH}$ and $d_{cc}$ to $U_{\widetilde{0}}$, are bi-Lipschitz equivalent.
\end{proposition}

\begin{proof}
For any $x_1\in\R^m, x_2,y_2\in\R^{n-m}$, and $t\in\R$ one can check directly from the formula for the Koranyi norm that,
\begin{equation}\label{korbnd}d_{\Hn}((x_1,x_2,0,y_2,t),0)\geq d_{\Hn}((0,x_2,0,y_2,t),0).\end{equation}
Now, as mentioned earlier, it is easy to check that $\omega(\V,U_{\widetilde{0}})=0$ so that $\V$ and $U_{\widetilde{0}}$ commute, and moreover, for $q\in\V$ and $p\in U_{\widetilde{0}}$, $qp=q+p$.
In particular, if $p,p'\in U_{\widetilde{0}}$ it follows that 
\begin{align*}
d_{\VH}(p',p)&=\inf_{q\in\V}d_{cc}(qp',p)\\
&=\inf_{q\in\V}d_{cc}(p^{-1}qp',0)\\
&=\inf_{q\in\V}d_{cc}(q+p^{-1}p',0)\\
&\simeq \inf_{q\in\V}d_{\Hn}(q+p^{-1}p',0)\\
&=d_{\Hn}(p^{-1}p',0)\simeq d_{cc}(p',p),
\end{align*}
where the first equality in the last line follows from \eqref{korbnd}. This completes the proof of the proposition.
\end{proof}

\begin{corollary}\label{embedding}
The map $\iota:(\He^{n-m},d_{cc,\He^{n-m}})\to(\Vp, d_{\VH})$ given by $\iota(x,y,t)=(0,x,0,y,t)$ is a bi-Lipchitz embedding.
\end{corollary}
\begin{proof}
It is clear that $\iota:\He^{n-m}\to U_{\widetilde{0}}\subset\Vp$ is bijective. By Proposition~\ref{restriction},
\begin{align*}
d_{\VH}(\iota(x,y,t),\iota(u,v,s)) &=d_{\VH}((0,x,0,y,t),(0,u,0,v,s))\\
&\simeq d_{cc}((0,x,0,y,t),(0,u,0,v,s))\\
&\simeq d_{cc,\He^{n-m}}((x,y,t),(u,v,s)). \qedhere \end{align*}
\end{proof}

 Proposition~\ref{restriction} and its corollary, do not hold for $\widetilde{y}\neq 0$. In particular, for $\widetilde{y}\neq 0$, the natural bijection of $\He^{n-m}$ onto the orbit $U_{\widetilde{y}}$ is not a bi-Lipschitz, embedding.
 Indeed, if $\widetilde{y}\neq 0$ and $p=(0,x,\widetilde{y},y,0), q=(0,u,\widetilde{y},v,0)\in U_{\widetilde{y}}$, we have
 \begin{equation}\label{KordistinU}
 d_{\Hn}(p,q)=\left[(|x-u|^2+|y-v|^2)^2+4(u\cdot y-x\cdot v)^2\right]^{1/4},\end{equation}
 whereas,
 \begin{multline*}
 d_{\VH}(p,q)\simeq \inf_{p'\in\V}d_{\Hn}(p'p,q)\\
 =\inf_{\widetilde{x}\in\R^m} \left\lVert \left(\widetilde{x},x-u,0,y-v,-\widetilde{x}\cdot\widetilde{y}-\frac{1}{2}(x\cdot v-y\cdot u) \right) \right\rVert_{\Hn}.
 \end{multline*}
 In particular, choosing $\widetilde{x}=-\frac{1}{2}(x\cdot v- y\cdot u)\frac{\widetilde{y}}{|\widetilde{y}|^2}$ gives the upper bound
 \[d_{\VH}(p,q)\lesssim \left[\frac{1}{4}(x\cdot v- y\cdot u)^2+|x-u|^2+|y-v|^2\right]^{1/2}. \]
 Comparing with \eqref{KordistinU} one sees that $d_{\VoH}\lfloor_{U_{\widetilde{y}}}$ cannot be bi-Lipschitz equivalent to $d_{\Hn}\lfloor_{U_{\widetilde{y}}}$, and therefore to $d_{cc}\lfloor_{U_{\widetilde{y}}}$.

 We expect the space $\VH$ to behave in an analogous way to the Grushin plane, $\G$, in that the metric should be Riemannian away from the critical subspace $U_{\widetilde{0}}$ and extend as a Carnot-Carath\'{e}odory metric to $U_{\widetilde{0}}$. We were unable to prove this in general, so it remains an interesting problem to check if $(\Vp,d_{\VH})$ is isometrically equivalent (or at least bi-Lipschitz equivalent) to a non equi-regular Carnot-Carath\'{e}odory space. In the specific case $m=1$ this is exactly true as we will see in the following section were we state this formally and give a sketch of the proof.

 \subsection{Vertical subgroups of co-dimension one}
 
 Consider the manifold $\R^{2n}$, with typical point denoted $(w, u_1,\ldots, u_{n-1}, w_1,\ldots, w_{n-1},\tau)$, and frame comprised of the vector fields
 \begin{equation}\label{Vfields}
 \Delta=\begin{cases}
 W=\frac{\partial}{\partial w}\\
 U_j=\frac{\partial}{\partial u_j}-\frac{w_j}{2}\frac{\partial}{\partial \tau}, & j=1,\ldots, n-1\\
 W_j=\frac{\partial}{\partial w_j}+\frac{u_j}{2}\frac{\partial}{\partial \tau}, & j=1,\ldots, n-1\\
 T=-w\frac{\partial}{\partial \tau}.
  \end{cases}
 \end{equation}
These vector fields span the entire tangent plane at every point outside of the critical $(2n-1)$-plane $\{w=0\}$, thus by declaring it orthonormal there, it induces a Riemannian distance on $\R^{2n}\setminus\{w=0\}$. Moreover, since $[U_j,W_j]=[T,W]=\frac{\partial}{\partial \tau}$, this metric can be extended to a Carnot-Caratheodory metric, $d_{\Delta}$, on all of $\R^{2n}$. Note that in the case $n=1$, the frame $\Delta$ consist only of the vector fields $V$ and $T$ and therefore the space $(\R^2, d_{\Delta})$ coincides with the Grushin plane.

\begin{proposition}
The space $(\Vp,\dvh)$ is isometric to $(\R^{2n},d_{\Delta})$.
\end{proposition}

\begin{proof}[Sketch of proof]
It is clear that, as sets, $\Vp$ and $\R^{2n}$ can be identified, so we consider the map $\PVp$ as a map from $\Hn$ to $\R^{2n}$. Firstly, we use the analytic change of variables in $\Hn$
\[\Psi(z,t)=[z,t]=(z,t+\frac{1}{2}\omega(\pi_V(z),\pi_{V^{\perp}}(z))).\]
Under this change of variables the horizontal vector fields become
\begin{equation}\label{modvfields}
\begin{cases}
\widetilde{X}_1=\frac{\partial}{\partial x_1}-y_1\frac{\partial}{\partial t}\\
\widetilde{Y}_1=\frac{\partial}{\partial y_1}\\
\widetilde{X}_j=\frac{\partial}{\partial x_j}-\frac{y_j}{2}\frac{\partial}{\partial t}, & j=2,\ldots, n\\
\widetilde{X}_j=\frac{\partial}{\partial y_j}+\frac{x_j}{2}\frac{\partial}{\partial t}, & j=2,\ldots, n,\\
\end{cases}
\end{equation}
and the projection map becomes 
\[\Phi_{\Vp}(z,t)=\PVp[z,t]=(\pi_{V^{\perp}}(z),t).\]
The differential of this map is easily computed to be the constant matrix
\[\Phi_{\Vp*}=\begin{pmatrix}
0&0\\
0&\textbf{I}
\end{pmatrix},\]
where $\textbf{I}$ is the $(2n)\times(2n)$ identity.  Hence, the push forward of the horizontal vector fields are 
\begin{equation}
\begin{cases}
X_1=\Phi_{\Vp*}\widetilde{X}_1=-y_1\frac{\partial}{\partial t}\\
Y_1=\Phi_{\Vp*}\widetilde{Y}_1=\frac{\partial}{\partial y_1}\\
X_j=\Phi_{\Vp*}\widetilde{X}_{j}=\frac{\partial}{\partial x_{j}}-\frac{y_{j}}{2}\frac{\partial}{\partial t}& j=2,\ldots, n\\
Y_j=\Phi_{\Vp*}\widetilde{Y}_{j}=\frac{\partial}{\partial y_{j}}+\frac{x_{j}}{2}\frac{\partial}{\partial t}& j=2,\ldots, n.\\
\end{cases}
\end{equation}
Note that these coincide exactly with \eqref{Vfields}, therefore if $\Gamma:[0,1]\to\Hn$ is a horizontal path in $\Hn$, then $\gamma=\PVp\circ\Gamma$ is a horizontal path in $(\R^{2n}, d_{\Delta})$. Indeed, $\Gamma$ is horizontal in $\Hn$ if there are integrable functions $a_j, b_j:[0,1]\to\R$ such that 
\[\dot{\Gamma}=\sum_{j=1}^na_j\widetilde{X}_j+b_j\widetilde{Y}_j,\]
therefore 
\[\dot{\gamma}=\Phi_{\Vp*}\dot{\Gamma}=a_1T+b_1W+\sum_{j=2}^{n}a_{j}X_j+b_{j}Y_j.\]
It follows that $\gamma$ is horizontal in $(\R^{2n}, d_{\Delta})$ and moreover,
\[\Lambda_{\Delta}(\gamma)=\int_0^1[\sum_{j=1}^n a_j^2+b_j^2]^{1/2}ds=\Lambda_{\Hn}(\Gamma).\]
This tells us that given $p,p'\in\Vp=\R^{2n}$, every $\Hn$-horizontal path between $\V p$ and $\V p'$ induces a $\Delta$-horizontal path between $p,p'\in\R^{2n}$ of the same length. Thus
\[d_{\Delta}(p,p')\leq \inf\{d_{cc}(qp,p'):q\in\V\}=\dvh(p,p').\]

Now we aim to show that every horizontal path in $(\R^{2n},\Delta)$ between $p,p'$ has a $\Hn$-horizontal lift between $\V p$ and $\V p'$ of the same length. This would imply $\dvh(p,p')\leq d_{\Delta}(p,p')$ and complete the proof.

To this end, let $\gamma=(w, u_1,\ldots,u_{n-1}, w_1,\ldots, w_{n-1}, \tau):[0,1]\to\R^{2n}$ be a horizontal path in $(\R^{2n},\Delta)$ with 
\[\dot{\gamma}=aW+\sum_{j=1}^{n-1}a_jW_j+b_jU_j+bT.\]
Put \[u(s)=u_0+\int_0^s b(\sigma)d\sigma,\]
where $u_0$ is arbitrarily chosen, so that $u:[0,1]\to\R$ is continuous and $\dot{u}(s)=b(s)$.
Then, set \[\Gamma(s)=(u(s),w(s),u_1(s),\ldots,v_{n-1}(s),\tau(s)).\]
It follows that $\Phi_{\Vp}(\Gamma)=\gamma$ and
\[\dot{\Gamma}=\dot{u}\frac{\partial}{\partial u}+\dot{\gamma}=b(\frac{\partial}{\partial u}-w\frac{\partial}{\partial \tau})+a\frac{\partial}{\partial w}+\sum_{j=1}^{n-1}a_jW_j+b_j U_j,\]
so $\Gamma$ is a horizontal path in $\Hn$ between the fibers $\Phi_{\Vp}^{-1}(0,\gamma(0))$ and $\Phi_{\Vp}^{-1}(0,\gamma(1))$. Furthermore,
\[\Lambda_{\Hn}(\Gamma)=\int_{0}^1\left[\sum_{j=1}^{n-1}a_j^2(s)+b_j^2(s)+a^2(s)+b^2(s)\right]^{1/2}ds=\Lambda_{\Delta}(\gamma),\]
and this completes the proof.
\end{proof}

Note that whenever $n>1$ the vector fields  $\{U_j, W_j: j=1,\ldots, n-1\}$ give rise to the embedded copy of $\He^{n-1}$ in $(\Vp,\dvh)$ that was mentioned in last section.  On the other hand, as mentioned above, when n=1 the frame $\Delta$ only consist of $V$ and $T$, and the Carnot-Caratheodory manifold $(\Vp,\dvh)$ is exactly the Grushin plane $\G$ with the embedded copy of ``$\He^0$" corresponding to the critical line. This last fact has been well known and used in conjunction with the right coset projections in the first Heisenberg group to solve certain iso-perimetric problems in the Grushin plane by projecting Heisenberg geodesics via $\PVp$ (\cite{HtoG}).



\section{Dimension distortion by right coset projections in \texorpdfstring{$\Hn$}{Heisenberg groups}} \label{dimensiondistortion}
We now have the appropriate setup to study dimension distortion by right coset projections. We have a family of 1-Lipschitz maps $\big\{\PVp:(\Hn,d_{cc})\to(\Vp,\dvh): V\in\Gh\big\}$ and would like to study the generic dimension of the sets $\PVp(A)$ for a given Borel set $A\subset\Hn$. First we note that since the maps are Lipschitz, the upper bound $\dim_{\VH}\PVp(A)\leq \dim_{\Hn}(A)$ holds trivially for all $\V$. Therefore, our main result focuses on almost sure dimension lower bounds. As we will see in the proof of the main result, lower bounds for the Euclidean Hausdorff dimension of projections will help us obtain lower bounds for their dimension with respect to the metric $\dvh$.

For any Borel subset $A$ of a complete separable metric space $(X,d)$, the Hausdorff dimension $\dim A$ of $A$ can be characterised using energy: $\dim A$ is the supremum over all $s \geq 0$ such that there exists a compactly supported probability measure $\mu$ on $A$ with 
\[ I_s(\mu, d) := \int \int d(x,y)^{-s} \, d\mu(x) \, d\mu(y) < \infty. \] 

The first four lemmas will show that for dimension lower bounds the a.e.~behaviour of projections with respect to the right coset metric is the same as with respect to the Euclidean metric.   

\begin{lemma} \label{locallipschitz} For fixed $V \in G_h(n,m)$, the identity map from $\left( \mathbb{V}^{\perp}, d_{\mathbb{V} \backslash \mathbb{H}^n } \right)$ to $\left( \mathbb{V}^{\perp}, d_E \right)$ is locally Lipschitz. 
\end{lemma} 
\begin{proof} Fix $R>0$ and $(z,t), (\zeta, \tau) \in \mathbb{V}^{\perp} \cap B_E(0,R)$. To prove
\[ d_E((z,t), (\zeta, \tau) ) \lesssim_R d_{\mathbb{V} \backslash \mathbb{H}^n }( (z,t), (\zeta, \tau) ), \]
it suffices to show that 
\begin{equation} \label{star5} \left\lvert z-\zeta \right\rvert+ \left\lvert t-\tau \right\rvert \lesssim_R \left\lvert z+w-\zeta \right\rvert + \left\lvert t-\tau+ \frac{1}{2} \omega(z,\zeta) - \frac{1}{2} \omega(z+\zeta, w) \right\rvert^{1/2}, \end{equation}
uniformly for all $w \in V$. If $|t-\tau| \leq 2R|z-\zeta|$ then \eqref{star5} follows from orthogonality, using only the first term in the right hand side.  Hence it may be assumed that 
\[ \left\lvert t-\tau \right\rvert \geq 2R \left\lvert z - \zeta \right\rvert. \]
If $\left\lvert w \right\rvert \geq \frac{ \left\lvert t- \tau \right\rvert}{4R}$ then \eqref{star5} again follows from orthogonality, so it may be assumed that 
\[ \lvert w \rvert \leq \frac{ \left\lvert t- \tau \right\rvert }{4R}. \]
Thus 
\begin{align*} \left\lvert t-\tau + \frac{1}{2} \omega(z, \zeta ) - \frac{1}{2} \omega(z+\zeta, w) \right\rvert &\geq \lvert t- \tau \rvert - \frac{R}{2} \lvert z-\zeta \rvert - R |w| \\
&\geq \frac{|t-\tau|}{2} \\
&\gtrsim_R |t-\tau|^2. \end{align*}
Taking square roots gives \eqref{star5}, and therefore proves the lemma. \end{proof}
The following lemma gives a sufficient condition under which the preceding inequality can be reversed.
\begin{lemma} \label{reverselipschitz} Fix $V \in G_h(n,m)$ and $(z,t), (\zeta, \tau) \in \mathbb{V}^{\perp}$.  If
\[ \lvert(z,t)\rvert, \lvert (\zeta,\tau) \rvert \leq C, \]
and there exists a unit vector $e \in V$ such that
\[ \left\lvert\omega(z+\zeta, e)\right\rvert \geq c >0, \] then 
\[  d_{\mathbb{V} \setminus \mathbb{H}^n }\left((z,t), (\zeta,\tau)\right) \lesssim_{c,C} d_E\left((z,t) ,(\zeta,\tau)\right). \]
\end{lemma}
\begin{proof} By definition, \begin{multline} \label{cosetdistance} d_{\mathbb{V} \setminus \mathbb{H}^n }((z,t) , (\zeta,\tau))  \\
\sim \inf_{w \in V} \left(|z-\zeta + w| + \left|t-\tau +\frac{1}{2} \omega(z, \zeta) +  \frac{1}{2} \omega(w, z+\zeta) \right|^{1/2}\right). \end{multline}
The point
\begin{equation} \label{pointchoice} w = \frac{-2\left(t-\tau+\frac{1}{2} \omega(z,\zeta) \right)e}{\omega(e,z+\zeta)}, \end{equation}
lies in $V$ and satisfies $|w| \lesssim_{c,C} d_E((z,t), (\zeta,\tau))$. Putting the $w$ from \eqref{pointchoice} into \eqref{cosetdistance} makes the second term vanish, and so
\[ d_{\mathbb{V} \setminus \mathbb{H}^n }((z,t) , (\zeta,\tau)) \lesssim_{c,C} d_E((z,t), (\zeta,\tau)). \qedhere \] \end{proof}

\begin{lemma} \label{equivalence} Fix $\beta \geq 0$, $n \geq 1$, $m\in \{1, \dots, n\}$ and $\alpha \in [2, 2n+2)$. The following two statements are equivalent.
\begin{enumerate}[(i)]
\item For any Borel set $A \subseteq \mathbb{H}^n$ with $\dim_{\mathbb{H}^n} A > \alpha$,
\[ \dim_{\mathbb{V} \setminus \mathbb{H}^n } P_{\mathbb{V}^{\perp}}^R(A)  \geq \beta \quad \text{ for a.e. } V \in G_h(n,m). \]
\item For any Borel set $A \subseteq \mathbb{H}^n$ with $\dim_{\mathbb{H}^n} A > \alpha$, 
\[ \dim_E P_{\mathbb{V}^{\perp}}^R(A)  \geq \beta \quad \text{ for a.e. } V \in G_h(n,m). \] \end{enumerate} \end{lemma}
\begin{proof} The implication $\text{(ii)} \Rightarrow \text{(i)}$ follows directly from Lemma~\ref{locallipschitz}, so assume that (i) holds. Let $A \subseteq \mathbb{H}^n$ be a compact set with $\dim_{\mathbb{H}^n} A > \alpha \geq 2$. 
Let $\mu$ be a Borel probability measure on $A$ with
\[ \mu(B_{\mathbb{H}^n}((z,t), r ) ) \lesssim r^s \quad \text{for all } (z,t) \in \mathbb{H}^n \text{ and } r>0, \]
where $2 \leq \alpha <  s < \dim_{\mathbb{H}^n} A$. Fix $s_0 >0$ with 
\begin{equation} \label{s0defn}  s_0 < \min\{ s-2, m\}. \end{equation}
By a similar covering argument to the proof of Theorem~1.1 in \cite{Hproj2},
\begin{equation} \label{energy47}  
\int_A \int_A \int_{G_h(n,m)} \frac{1}{|\pi_V(z-\zeta)|^{s_0}} \, d\mu_{n,m}(V) \, d\mu(z,t) \, d\mu(\zeta,\tau)< \infty, \end{equation}
where the inner integral is bounded using the inequality from the proof of Theorem~1.2 in \cite{Hproj2}. Let $\mathcal{U}$ be a nonempty open subset of $G_h(n,m)$ such that there exists a continuously varying orthonormal basis $\{v_1(V), \dotsc, v_m(V) \}$ for $V$ as $V$ varies over $\mathcal{U}$ (which exists e.g.~by Gram-Schmidt). By covering $G_h(n,m)$ with a finite number of such sets, it will suffice to show that 
\[ \dim_E P_{\mathbb{V}^{\perp}}(A)  \geq \beta \quad \text{ for a.e. } V \in \mathcal{U}. \]
Coordinate-wise multiplication by $i$ from $\mathbb{C}^n$ to $\mathbb{C}^n$ is a linear map in the complex unitary group $U(n)$, and since $\mu_{m,n}$ is $U(n)$-invariant (see \cite{Hproj2}), \eqref{energy47} yields
\begin{equation} \label{pause42} |\pi_{iV}(z-\zeta)| > 0, \end{equation}
for $\mu \times \mu \times \mu_{n,m}$ almost every $((z,t), (\zeta, \tau), V) \in A \times A \times \mathcal{U}$. Let $\epsilon >0$; the preceding statement gives a $\delta>0$ such that
\[ (\mu \times \mu \times \mu_{n,m}) \left\{ ((z,t), (\zeta, \tau), V) \in A \times A \times \mathcal{U} : \left\lvert \pi_{iV}(z-\zeta) \right\rvert \leq \delta \right\} < \epsilon. \] 
By Fubini, this in turn implies that 
\begin{equation} \label{squareroot} \mu_{m,n}( \mathcal{U}_0) \geq \mu_{n,m}(\mathcal{U})- \sqrt{\epsilon}, \end{equation}
where 
\begin{equation} \label{unought} \mathcal{U}_0 := \left\{ V \in \mathcal{U} : (\mu \times \mu) \left\{ ((z,t), (\zeta, \tau)) \in A \times A : \left\lvert \pi_{iV }(z-\zeta) \right\rvert \leq \delta \right\} \leq \sqrt{\epsilon} \right\}. \end{equation}
Let 
\[ \mathcal{U}_0 = \bigcup_{k=1}^N \mathcal{U}_0^{(k)}, \]
be a finite, disjoint partition of $\mathcal{U}_0$ into nonempty sets $\mathcal{U}_0^{(k)}$ such that 
\begin{equation} \label{deltasquare} \left\lvert v_j(V) - v_j(V')\right\rvert < \delta^2 \quad \text{for all } V, V' \in \mathcal{U}_0^{(k)} \text{ and for all } j, k. \end{equation}
The definition of $\mathcal{U}_0$ in \eqref{unought} implies that for each $V \in \mathcal{U}_0$,
\[ \mu \left\{ (z,t) \in A : \left\lvert \pi_{iV}(z) \right\rvert > \delta/2 \right\} \geq 1 - \epsilon^{1/4}. \]
Hence for each $k$ there exists $V_k \in \mathcal{U}_0^{(k)}$ and a Borel set $B_k \subseteq A$ with 
\[ \mu(B_k)  \gtrsim 1 \quad \text{and} \quad \left\lvert \pi_{iV_k}(z) \right\rvert \gtrsim \delta \quad \text{ for all } (z,t) \in B_k. \]
Therefore for each $k$ there exists $j= j(k) \in \{1, \dotsc, m\}$, $\sigma= \sigma(k) \in \{0,1\}$ and a Borel set $A_k \subseteq A$ such that 
\begin{equation} \label{signed} \mu(A_k) \gtrsim 1 \quad \text{and} \quad \omega\left( z, (-1)^{\sigma} v_j(V_k) \right) \gtrsim \delta \quad \text{ for all } (z,t) \in A_k. \end{equation}
If $\delta$ is sufficiently small (which may be assumed), then by \eqref{deltasquare} and \eqref{signed},
\[ \left\lvert \omega\left( z+\zeta ,v_j(V) \right) \right\rvert \gtrsim \delta \quad \text{for all } (z,t), (\zeta, \tau) \in A_k, \quad V \in \mathcal{U}_0^{(k)} \text{ and } j=j(k). \]
Since each $V$ is isotropic, it then follows from Lemmas~\ref{locallipschitz} and~\ref{reverselipschitz} that 
\[  d_{\mathbb{V} \setminus \mathbb{H}^n }\left(P_{\mathbb{V}^{\perp}}(z,t) , P_{\mathbb{V}^{\perp}}(\zeta,\tau)\right) \sim_{\delta} d_E\left(P_{\mathbb{V}^{\perp}}(z,t) , P_{\mathbb{V}^{\perp}}(\zeta,\tau)\right), \]
for all $(z,t), (\zeta, \tau) \in A_k$ and $V \in \mathcal{U}_0^{(k)}$. Therefore 
\[  \dim_E (P_{\mathbb{V}^{\perp} }(A) ) \geq  \dim_E (P_{\mathbb{V}^{\perp} }(A_k) ) = \dim_{\mathbb{V} \setminus \mathbb{H}^n }(P_{\mathbb{V}^{\perp}}(A_k)), \]
for all $k$ and $V \in \mathcal{U}_0^{(k)}$. Applying (i) for each $k$ gives
\[ \dim_E (P_{\mathbb{V}^{\perp} }(A) ) \geq \beta, \]
for $\mu_{n,m}$-a.e.~$V \in \mathcal{U}_0$. But $\mu_{n,m}(\mathcal{U}_0) \geq \mu_{n,m}(\mathcal{U}) - \sqrt{\epsilon}$ by \eqref{squareroot}, so letting $\epsilon \to 0$ and covering $G_h(n,m)$ with a finite number of such sets $\mathcal{U}$ gives  
\[ \dim_E (P_{\mathbb{V}^{\perp} }(A) ) \geq \beta, \]
for a.e.~$V \in G_h(n,m)$. This proves that (i) and (ii) are equivalent. \end{proof}
The preceding lemma and the following one actually hold for $\alpha \geq 0$, but for small $\alpha$ this follows from Theorem~\ref{Hnbounda} (since the a.e.~lower bounds for dimension that follow from Theorem~\ref{Hnbounda} are the same in the smaller range of $\alpha$, and by Theorem~\ref{Hnbounda} they are both sharp). The proof of the following lemma is omitted since it is virtually identical to the previous one, except that $s-1$ is used in \eqref{s0defn} instead of $s-2$.
\begin{lemma} \label{equivalence2} Fix $\beta \geq 0$, $n \geq 1$, $m\in \{1, \dots, n\}$ and $\alpha \in [1, 2n+1)$. The following two statements are equivalent.
\begin{enumerate}
\item For any Borel set $A \subseteq \mathbb{H}^n$ with $\dim_E A > \alpha$, 
\[ \dim_{\mathbb{V} \setminus \mathbb{H}^n } P_{\mathbb{V}^{\perp}}^R(A)  \geq \beta \quad \text{ for a.e. } V \in G_h(n,m). \]
\item
For any Borel set $A \subseteq \mathbb{H}^n$ with $\dim_E A > \alpha$, 
\[ \dim_E P_{\mathbb{V}^{\perp}}^R(A)  \geq \beta \quad \text{ for a.e. } V \in G_h(n,m). \]
\end{enumerate}
\end{lemma} 

 We restate Theorem~\ref{Hnbounda} here.
\begin{reptheorem}{Hnbounda} For $1 \leq m \leq n$ and any Borel set $A \subseteq \mathbb{H}^n$,  
\begin{equation} \label{euclidean2}
\dim_E P_{\mathbb{V}^{\perp}}^L (A), \ \dim_E P_{\mathbb{V}^{\perp}}^R (A) \geq\max\left\{ \min\left\{\dim_E A, 2n-m \right\},\dim_E A-m\right\}
\end{equation}
for $\mu_{n,m}$-a.e.~$V \in G_h(n,m)$, and
\begin{multline} \label{koranyi2} \dim_{\mathbb{V} \backslash \mathbb{H}^n} P_{\mathbb{V}^{\perp}}^R (A)\\ 
\geq \max\left\{ \min \left\{ \max\left\{\frac{\dim_{\mathbb{H}^n} A}{2}, \dim_{\mathbb{H}^n} A -1 \right\},2n-m \right\}, \dim_{\mathbb{H}^n}A -m-1 \right\}
\end{multline}
for $\mu_{n,m}$-a.e.~$V \in G_h(n,m)$. If $\dim_E A \leq 2n-m$ then \eqref{euclidean2} is sharp, and if $\dim_{\mathbb{H}^n} A \leq 2n+1-m$ then \eqref{koranyi2} is sharp. 
\end{reptheorem}
\begin{remark} By Lemmas~\ref{equivalence} and~\ref{equivalence2}, and by the method of proof used, this theorem holds verbatim if the left hand sides of \eqref{euclidean2} and \eqref{koranyi2} are interchanged.
\end{remark}
\begin{proof}[Proof of Theorem~\ref{Hnbounda}] The cases $P_{\mathbb{V}^{\perp}}^L$ and $P_{\mathbb{V}^{\perp}}^R$ in \eqref{euclidean2} are equivalent since $P_{\mathbb{V}^{\perp}}^L(p) = -P_{\mathbb{V}^{\perp}}^R(-p)$. 
For the remainder of the proof, the notation $P_{\mathbb{V}^{\perp}}$ will therefore denote $P_{\mathbb{V}^{\perp}}^R$.

The quotient distance on $\mathbb{V}^{\perp}$ is defined through the identification of $\mathbb{V}^{\perp}$ with $\mathbb{V} \backslash \mathbb{H}^n$ explained in Section \ref{rightcoset}; the formula is given by 
\[ d_{\mathbb{V} \backslash \mathbb{H}^n}(p,q) = \inf_{q' \in \mathbb{V}} d_{cc}(q' p, q ), \quad \text{where } p,q \in \mathbb{V}^{\perp}. \]
Since the metric $d_{\Hn}$ is bi-Lipschitz equivalent to $d_{cc}$, one can set \[ d_{\mathbb{V} \backslash \mathbb{H}^n}'(p,q) = \inf_{q' \in \mathbb{V}} d_{\Hn}(q' p, q ),\]
and trivially obtain that $\dvh$ and $\dvh'$ are bi-Lipschitz equivalent. For ease of computation we use $\dvh'$ instead of $\dvh$, and to simplify notation we denote $\dvh'$ by $\dvh$ as well. 

It may be assumed without loss of generality that $A$ is bounded. Let $\mu$ be a measure on $A$ with Euclidean $s$-energy $I_s(\mu, d_E) < \infty$, where $s:= \min \left\{\dim_E A ,2n-m\right\} - \epsilon$ for an arbitrarily small $\epsilon >0$. Assume $s>0$ without loss of generality. By Fubini, the average energy of the pushforward measure is 
\begin{multline*} \int_{G_h(n,m)} I_s(P_{\mathbb{V}^{\perp}\#} \mu, d_E) \, d\mu_{n,m}(V) \\
= \int_{\mathbb{H}^n} \int_{\mathbb{H}^n} \int_{G_h(n,m)} d_E\left(P_{\mathbb{V}^{\perp}}(z,t), P_{\mathbb{V}^{\perp}}(\zeta,\tau) \right)^{-s}  \, d\mu_{n,m}(V) \, d\mu(z,t) \, d\mu(\zeta,\tau). \end{multline*}
To prove the Euclidean lower bound in the first part of the minimum of \eqref{euclidean2}, it suffices to show
\begin{equation} \label{muenergy2} \int_{G_h(n,m)} d_E\left(P_{\mathbb{V}^{\perp}}(z,t), P_{\mathbb{V}^{\perp}}(\zeta,\tau) \right)^{-s} \, d\mu_{n,m}(V) \lesssim d_E((z,t), (\zeta,\tau))^{-s}. \end{equation}
The first half of this proof will be essentially the same as the proof of Theorem~1.2 in \cite{Hproj2}. Let $B(0,R)$ be a Euclidean ball containing $A$. If $|z-\zeta| \geq \frac{\left|t-\tau \right|}{4R}$, then
\begin{align} \notag & \int_{G_h(n,m)} d_E\left(P_{\mathbb{V}^{\perp}}(z,t), P_{\mathbb{V}^{\perp}}(\zeta,\tau) \right)^{-s} \, d\mu_{n,m}(V)   \\
\notag &\quad \lesssim \int_{G_h(n,m)} \left|\pi_{V^{\perp}}(z) -\pi_{V^{\perp}}(\zeta) \right|^{-s} \, d\mu_{n,m}(V) \\
\label{tysonetal2} &\quad\lesssim_s |z-\zeta|^{-s} &&\text{(since $s<2n-m$)} \\ 
\notag &\quad \lesssim_R d_E((z,t), (\zeta,\tau))^{-s}; \end{align}
the Euclidean inequality used in \eqref{tysonetal2} is explained in \cite[pp.~584-585]{Hproj2}, and has a fairly straightforward proof. This proves \eqref{muenergy2} in the case where $|z-\zeta| \geq \frac{\left|t-\tau \right|}{4R}$. 

In the second case with $|z-\zeta| < \frac{\left|t-\tau\right|}{4R}$, Cauchy-Schwarz gives
\begin{align} \label{cauchyschwarz} &\left\lvert \omega \left( \pi_{V_{\theta}}(z), \pi_{V_{\theta}^{\perp}}(z)\right) - \omega \left( \pi_{V_{\theta}}(\zeta), \pi_{V_{\theta}^{\perp}}(\zeta)\right) \right\rvert \\
\notag &= \left\lvert \omega \left( \pi_{V_{\theta}}(z-\zeta), \pi_{V_{\theta}^{\perp}}(z)\right) - \omega \left( \pi_{V_{\theta}}(\zeta), \pi_{V_{\theta}^{\perp}}(\zeta-z)\right) \right\rvert \leq 2R |z-\zeta|. \end{align}
Hence
\begin{align*}  &\int_{G_h(n,m)} d_E\left(P_{\mathbb{V}^{\perp}}(z,t), P_{\mathbb{V}^{\perp}}(\zeta,\tau) \right)^{-s} \, d\mu_{n,m}(V)  \\
&\quad \lesssim \int_{G_h(n,m)} \left\lvert t-\tau - \frac{1}{2} \omega(\pi_V(z), \pi_{V^{\perp}}(z) ) + \frac{1}{2} \omega( \pi_V(\zeta), \pi_{V^{\perp}}(\zeta) ) \right\rvert^{-s} \, d\mu_{n,m}(V), \\
&\quad \lesssim \left(\left\lvert t-\tau \right\rvert - R\left\lvert z-\zeta \right\rvert \right)^{-s}, \\
&\quad \lesssim |t-\tau|^{-s} \\
&\quad \lesssim_R d_E((z,t), (\zeta,\tau))^{-s}. \end{align*}
This proves the Euclidean lower bound for the first term in the maximum of \eqref{euclidean2}, which finishes the proof of \eqref{euclidean2} in the case $\dim_E A\leq 2n$.

The lower bound 
\[\dim_{E}P_{\Vp} (A) \geq\dim_{E}(A)-m\]
actually holds for all $\Vp$, provided $\dim_{E}A>m+1.$ Since the previous bound is stronger whenever $\dim_{E}A\leq 2n$, we may assume, without loss of generality, that $\dim_{E}A> 2n$. In particular $\dim_{E}A>m+1$. From here the proof follows the same lines as the proof of the same lower bound for the $\Hn$-dimension of (left coset) vertical projections from \cite[Theorem~1.4]{Hproj2}. Given $V\in\Gh$ the set $\{U\in\Gh: U^{\perp}\cap V=\{0\}\}$ is open, nonempty and in particular has positive $\munm$ measure. This, together with Theorem~\ref{EuclideanSlice}, lets us pick for $\epsilon>0$, $U\in\Gh$ and $u \in U$ such that the map $\pi_{V^{\perp}}\lfloor_{U^{\perp}}:U^{\perp}\to V^{\perp}$ is injective, and $\dim_{E}[A\cap(\Up*u)]\geq  \dim_{E}A-m-\epsilon$. For this particular choice of $U$ and $u$, we will see that $P_{\Vp}\lfloor_{\Up*u}:\Up*u\to\Vp$ is a locally bi-Lipschitz bijection with respect to the Euclidean norm.

First we show injectivity. For any $q\in(\U^{\perp}*u)$, there exists a unique  $w_{U^{\perp}}\in U^{\perp}$ and $s>0$ such that $q=(w_{U^{\perp}},s)*(u,0)$. Let $q=(w_{U^{\perp}},s)*(u,0) \in(\Up*u)$ and $q'=(z_{U^{\perp}},t)*(u,0)\in(\Up*u)$ be such that $P_{\Vp}(q)=P_{\Vp}(q')$. Then we have 
\begin{multline} \label{eqproj} \left(\pi_{V^{\perp}}(w_{U^{\perp}}+u),s+\frac{1}{2}\omega(w_{U^{\perp}},u)-\frac{1}{2}\omega(\pi_V(w_{U^{\perp}}+u),\pi_{V^{\perp}}(w_{U^{\perp}}+u)\right) \\
= \left(\pi_{V^{\perp}}(z_{U^{\perp}}+u) ,t+\frac{1}{2}\omega(z_{U^{\perp}},u)-\frac{1}{2}\omega(\pi_V(z_{U^{\perp}}+u),\pi_{V^{\perp}}(z_{U^{\perp}}+u)\right). \end{multline}
The first coordinate tells us that $\pi_{V^{\perp}}(z_{U^{\perp}}+u)=\pi_{V^{\perp}}(w_{U^{\perp}}+u)$ which says $\pi_{V^{\perp}}(z_{U^{\perp}})=\pi_{V^{\perp}}(w_{U^{\perp}})$. By our choice of $U\in\Gh$ we get that $z_{U^{\perp}}=w_{U^{\perp}}$. Similarly, the second coordinate gives us that $t=s$ so injectivity follows.
To see that the map is surjective, for $(z,t)\in\Vp$ put $\zeta=(\pi_{V^{\perp}}\lfloor_{U^{\perp}})^{-1}(z-\pi_{V^{\perp}}(u))+u$, and $\tau=t+\frac{1}{2}\omega(\pi_{V}(\zeta),\pi_{V^{\perp}}(\zeta))$. It follows that $(\zeta,\tau)\in\Up*u$ and $P_{\Vp}(\zeta,\tau)=(z,t)$. This shows that the map is surjective, but also gives us a formula for the inverse which shows this inverse map is smooth. Hence $P_{\V}\lfloor_{\Up*u}$ is a smooth map with a smooth inverse, and it is therefore locally bi-Lipschitz with respect to the Euclidean metric. By the choice of $\U$,
\begin{multline*}
\dim_E  P_{\Vp}(A)\geq  \dim_EP_{\Vp}(A\cap(\Up*u))
=\dim_E [A\cap(\Up*u)] \geq \dim_{E}A-m-\epsilon. 
\end{multline*}
Since $\epsilon$ can be chosen arbitrarily small, this proves the lower bound in \eqref{euclidean2}.
The lower bound in \eqref{koranyi2} follows from Lemma~\ref{locallipschitz} and the Dimension Comparison Principle applied to the lower bound in \eqref{euclidean2}.
The Dimension Comparison Principle says that for any set $B \subseteq \mathbb{H}^n$,
\begin{equation} \label{dimcomparison} \max\{ \dim_E B, 2 \dim_E B - 2n \} \leq \dim_{\mathbb{H}^n} B \leq \min\{ 2 \dim_E B , \dim_E B +1 \}. \end{equation}
This  comparison principle, as stated here,  appears in \cite[Eq. 1.4]{Hproj2}, see \cite{BRSC} for the original proof in $\mathbb{H}$ and see \cite{BTW} for the proof in the more general case of Carnot groups. 

The sharpness of the Euclidean lower bound in \eqref{euclidean2} will be deduced from the sharpness of the Heisenberg lower bound in \eqref{koranyi2}. The sharpness of \eqref{koranyi2} for $\dim_{\mathbb{H}^n} A \leq 2n+1-m$ will be proved in two separate cases. For an example with any Heisenberg dimension in the range $[0,2]$, let $\alpha \in [0,2]$ and let $A$ be a compact subset of the vertical segment
\[ \{(e_1,s) \in \mathbb{R}^{2n} \times \mathbb{R} =  \mathbb{H}^n: s \in  [-1/4,1/4]\}, \]
such that $\dim_{\mathbb{H}^n} A = 2 \dim_E A = \alpha$, where $e_j$ is the $j$-th standard basis vector in Euclidean space. Let 
\[ \mathcal{U} = \{ V \in G_h(n,m) : |\omega(e_1, w)| > 1/2 \text{ for some } w \in V \text{ with } |w| \leq 1 \}. \]
Then $\mathcal{U}$ is a nonempty open set, and so $\mu_{n,m}\left(\mathcal{U}\right) >0$. For $(e_1, s), (e_1,t) \in A$ and $V \in \mathcal{U}$, there exists $w \in V$ with $|w| \leq 2|s-t|$ such that $\omega(e_1,w)=s-t$. Hence
\begin{align*} d_{\mathbb{V} \backslash \mathbb{H}^n }\left( P_{\mathbb{V}^{\perp}}(e_1,s), P_{\mathbb{V}^{\perp}}(e_1, t) \right) &\sim \inf_{w \in V} \left( |w|^4 + \left|s-t+ \omega(w, e_1) \right|^2 \right)^{1/4} \\
&\sim |s-t| \\
&\sim  d_{\mathbb{H}^n}((e_1,s),(e_1, t))^2. \end{align*}
It follows that 
\[ \dim_{\mathbb{V} \backslash \mathbb{H}^n} P_{\mathbb{V}^{\perp}} (A) = \frac{\dim_{\mathbb{H}^n} A}{2} \quad \text{for $V \in \mathcal{U}$.} \]
This shows that \eqref{koranyi2} is sharp for $\dim_{\mathbb{H}^n} A \leq 2$, which by \eqref{dimcomparison} and Lemma \ref{locallipschitz}, implies the sharpness of \eqref{euclidean2} for $\dim_E A \leq 1$.

For an example with any Heisenberg dimension in the range $(2, 2n+1-m]$, 
 let $A$ be a set in $\mathbb{H}^n$ 
with
\begin{equation} \label{productconditions} A= \mathcal{C}_{\alpha} \times I, \quad \dim_E A = \alpha+1, \quad \dim_{\mathbb{H}^n} A = \alpha+2, \end{equation}
where $\mathcal{C}_{\alpha} \subseteq \mathbb{C}^n$ is a compact set of Euclidean dimension $\alpha \in [0, 2n-1]$ and $I$ is a compact interval of positive length (this is based on Case~2 from \cite[Section~4.1]{BRSC}; see also Appendix~\ref{appendix2} for an explicit construction). Then
\[ \dim_E P_{\mathbb{V}^{\perp}}(A) \leq \alpha+1 = \dim_{\mathbb{H}^n} A - 1, \]
for all $V \in G_h(n,m)$. By varying $\alpha$ and using Lemma~\ref{equivalence}, this shows that \eqref{koranyi2} is sharp if $2 \leq \dim_{\mathbb{H}^n} A \leq 2n+1-m$. By the Dimension Comparison Principle (see \eqref{dimcomparison}), this implies that the lower bound of \eqref{euclidean2} is sharp for $\dim_E A \leq 2n-m$. \end{proof}
Theorem~\ref{LeftCoset} now follows directly from Theorem~\ref{Hnbounda} and the Dimension Comparison Principle. 

To finish this section, we 
prove the relations in \eqref{conjecturerelations}.
\begin{proof}[Sharpness of conjectures and proof of implications in \eqref{conjecturerelations}] The equivalence of Conjecture~\ref{conjB} and Conjecture~\ref{conjF} follows from Theorem~\ref{Hnbounda} and Lemma~\ref{equivalence}. 
The equivalence of Conjecture~\ref{conjD} and Conjecture~\ref{conjE} is similar. The implication Conjecture~\ref{conjA} $\Rightarrow$ Conjecture~\ref{conjC} follows from dimension comparison on the left hand side. The two vertical implications in \eqref{conjecturerelations} both follow directly from dimension comparison on the right hand side.

The sharpness of Conjectures~\ref{conjC} and~\ref{conjF} follows from the same example as in Theorem~\ref{Hnbounda}; which works in the slightly larger range. The sharpness of the other conjectures is a consequence of the relations in \eqref{conjecturerelations}. \end{proof}

\section{An improved bound in \texorpdfstring{$\mathbb{H}$}{the first Heisenberg group}} \label{first}

In this section we prove a result for Euclidean dimension distortion under projections in $\mathbb{H}=\He^1$, which for $n=1$ improves Theorem~\ref{Hnbounda} in a small range. 

Since the family of nontrivial horizontal subgroups in $\mathbb{H}$ is the one dimensional family of lines in $\mathbb{C} \times \{0\}$ through the origin, the symbols $V_{\theta}$ and $\mathbb{V}_{\theta}$ will be used for the 1-dimensional subspaces containing $(\cos \theta , \sin \theta)$ and $(\cos \theta, \sin \theta, 0)$, respectively. In this section, the notation $P_{\mathbb{V}_{\theta}^{\perp}}$ will indicate either $P_{\mathbb{V}_{\theta}^{\perp}}^L$ or $P_{\mathbb{V}_{\theta}^{\perp}}^R$; the left/right designation will only be used if necessary. The proofs of Lemmas~\ref{transversality} and~\ref{paradox} use the right coset formula in computations, but by symmetry this is inessential.

The following (standard) lemma essentially says that the family of vertical projections obeys a weak version of transversality with respect to the Euclidean metric. 

\begin{lemma} \label{transversality} Let $R>0$. For any distinct $(z,t)$, $(\zeta, \tau) \in \mathbb{H} \cap B_E(0,R)$ and any $\delta \in (0,1)$, the set 
\[ \left\{ \theta \in [0,\pi) : d_E\left( P_{\mathbb{V}_{\theta}^{\perp}}(z,t),  P_{\mathbb{V}_{\theta}^{\perp}}(\zeta,\tau) \right) < \delta \right\} \]
is contained in $\lesssim 1$ intervals of length $\lesssim_R \frac{\delta}{d_E((z,t), (\zeta, \tau) )}$. \end{lemma}
\begin{proof} Suppose that $|z-\zeta| \geq \frac{|t-\tau|}{2R}$. Then 
\begin{multline*} \left\{ \theta \in [0,\pi) : d_E\left( P_{\mathbb{V}_{\theta}^{\perp}}(z,t),  P_{\mathbb{V}_{\theta}^{\perp}}(\zeta,\tau) \right) < \delta \right\} \\
 \subseteq  \left\{ \theta \in [0,\pi) : \left\lvert \pi_{V_{\theta}^{\perp}}(z) - \pi_{V_{\theta}^{\perp}}(\zeta) \right\rvert < \delta \right\}. \end{multline*}
By scaling, rotation and by transversality of the zeroes of $\theta \mapsto \sin \theta$, the right hand side is contained in at most 2 intervals of length $\lesssim \frac{\delta}{\lvert z-\zeta\rvert} \lesssim_R \frac{\delta}{d_E((z,t), (\zeta, \tau) )}$. This proves the lemma in case $|z-\zeta| \geq \frac{|t-\tau|}{2R}$. 

Now suppose that $|z-\zeta| < \frac{|t-\tau|}{2R}$. In this case if $|t-\tau| < 2\delta$ the lemma is trivial, so assume $|t-\tau| \geq 2\delta$. Then 
\begin{multline} \label{lastcase} \left\{ \theta \in [0,\pi) : d_E\left( P_{\mathbb{V}_{\theta}^{\perp}}(z,t),  P_{\mathbb{V}_{\theta}^{\perp}}(\zeta,\tau) \right) < \delta \right\} \subseteq \\
 \left\{ \theta \in [0,\pi) : \left\lvert t-\tau - \frac{1}{2} \omega \left( \pi_{V_{\theta}}(z), \pi_{V_{\theta}^{\perp}}(z)\right)+ \frac{1}{2} \omega \left( \pi_{V_{\theta}}(\zeta), \pi_{V_{\theta}^{\perp}}(\zeta)\right)  \right\rvert < \delta \right\}. \end{multline}
Similarly to \eqref{cauchyschwarz}, Cauchy-Schwarz gives
\begin{equation} \label{cliche} \left\lvert t-\tau - \frac{1}{2} \omega \left( \pi_{V_{\theta}}(z), \pi_{V_{\theta}^{\perp}}(z)\right)+ \frac{1}{2} \omega \left( \pi_{V_{\theta}}(\zeta), \pi_{V_{\theta}^{\perp}}(\zeta)\right)  \right\rvert \geq \frac{\left\lvert t-\tau \right\rvert }{2}. \end{equation}
Since $|t-\tau| \geq 2\delta$, the set in the right hand side of \eqref{lastcase} is empty, and this finishes the proof.  \end{proof}

The following lemma is the main result of this section, which will be converted to a projection theorem via a standard technique.  The proof will only be sketched since it is similar to the case of Euclidean projections in $\mathbb{R}^3$ \cite{OrpVen}, and also to the Korányi metric case of left projections in $\mathbb{H}$ \cite{THarris2}; the main emphasis will be on the steps which differ from \cite{THarris2}.

\begin{lemma} \label{paradox} Fix $s>1$, and let $\nu$ be a compactly supported Borel measure on $\mathbb{H}$ such that 
\[ \sup_{\substack{ x \in \mathbb{H} \\ r >0 } } \frac{ \nu(B_E(x,r) ) }{r^s} < \infty. \]
For any $\kappa > 
\frac{2(s-1)}{3}$, there exist $\delta_0$, $\eta >0$ such that 
\begin{equation} \label{russell} \nu\left\{ x \in \mathbb{H} : \mathcal{H}^1 \left\{ \theta \in [0,\pi) : P_{\mathbb{V}_{\theta}^{\perp} \#}\nu \left( B_E\left(P_{\mathbb{V}_{\theta}^{\perp}}(x), \delta\right) \right) \geq \delta^{s-\kappa} \right\} \geq \delta^{\eta} \right\} \leq \delta^{\eta}, \end{equation}
for all $\delta \in (0,\delta_0)$. 
\end{lemma}
\begin{proof}  Assume without loss of generality that $\nu$ is supported in the unit ball, and that $\kappa < s-1$. Choose $\eta$ with
\begin{equation}  \label{etaassumption} 0< \eta \ll   \kappa - \frac{2(s-1)}{3}, \end{equation}
where the right hand side is positive by assumption. Define $A \lessapprox B$ to mean $A \lesssim \delta^{-O(\eta)} B$, and write $A \approx B$ if $A \lessapprox B$ and $B \lessapprox A$.

Let $Z$ be the set of $x$'s occurring in \eqref{russell}. The argument that follows works for any $\delta>0$ sufficiently small, so we assume $\delta_0$ has been suitably chosen and $\delta \in (0, \delta_0)$. For any such $\delta$, dyadic pigeonholing gives a set $Z' \subseteq Z$ with $\nu(Z') \approx \nu(Z)$ and a fixed dyadic number $t$ with $\delta \leq t \lesssim 1$, such that for each $x \in Z'$ there are three sets $H_1(x), H_2(x), H_3(x) \subseteq [0,\pi)$ that are $\approx 1$-separated for each $x$, each with $\mathcal{H}^1$-measure $\approx 1$, such that 
\begin{equation} \label{pushfail}  \nu \left(  A_E(x,t,2t)  \cap  P_{\mathbb{V}_{\theta}^{\perp}}^{-1} \left(B_E(P_{\mathbb{V}_{\theta}^{\perp}}(x), \delta) \right) \right) \gtrapprox \delta^{s-\kappa} \quad \text{for all } \theta \in H_i(x), \end{equation}
where $A_E(x,t,2t)$ is the Euclidean annulus around $x$ of inner radius $t$ and outer radius $2t$. This pigeonholing step is virtually identical to those in \cite{OrpVen} and \cite{THarris2} (where more details are provided). 

Let 
\begin{equation} \label{alphadefn} \alpha = \frac{s-1-\kappa+O(\eta)}{s-1}, \end{equation}
and let 
\[ \Lambda = 
\begin{cases} \Big\{ (x,x_1,x_2,x_3) \in Z' \times \left( \mathbb{H} \right)^3 : \\ \qquad d_E(z_2, \ell(z_1, z_3) ) \geq \delta^{\alpha} \text{ if } |z-z_1|, |z-z_3| \geq t/2 \Big\}, & t\gtrapprox \delta^{\alpha}, \\
\\
Z' \times \left( \mathbb{H} \right)^3, & t\lessapprox \delta^{\alpha}, \end{cases} \]
where $x=(z,\tau)$ and $\ell(z,w)$ is the line through $z$ and $w$ in $\mathbb{R}^2$. 
The lemma will follow from the outer two parts of 
\begin{equation} \label{maininequality} \nu(Z) t^3 \delta^{3(s-\kappa-1)} \lessapprox \nu^4\left\{ (x,x_1,x_2,x_3) \in \Lambda : x \sim_i x_i \text{ for all } i \right\} \\
\lessapprox \begin{cases} \delta^{(1-\alpha)s} 
t^s & t \gtrapprox \delta^{\alpha}, \\
t^{3s} & t \lessapprox \delta^{\alpha}. \end{cases} \end{equation}
where $x \sim_i x_i$ means that 
\begin{equation} \label{equivrelation} t \leq d_E(x,x_i) < 2t \quad \text{and} \quad  d_E\left( P_{\mathbb{V}_{\theta}^{\perp}}(x), P_{\mathbb{V}_{\theta}^{\perp}}(x_i) \right) < \delta, \end{equation}
for some angle $\theta \in H_i(x)$. 

The lower bound of \eqref{maininequality} essentially follows by fixing $x \in Z'$, establishing the lower bound $t \delta^{s-\kappa-1}$ on the $\nu$-measure of the set of $x_i$'s satisfying $x \sim_i x_i$, integrating over $x_1$, $x_2$ and $x_3$ to get $t^3 \delta^{3(s-\kappa-1)}$, integrating over $x \in Z'$ and using $\nu(Z) \approx \nu(Z')$. This argument is similar to the one in \cite{OrpVen}, except that here as in \cite{THarris2} the points $(x,x_1, x_2, x_3)$ have the additional requirement that they must be in $\Lambda$. For the lower bound $t \delta^{s-\kappa-1}$ on the $\nu$-measure of the set of $x_i$'s satisfying $x \sim_i x_i$, the proof proceeds by sorting the points $x_i$ according to the interval $I_k$ of length $\delta/t$ containing the corresponding angle $\theta$ in \eqref{equivrelation}, using \eqref{pushfail} to bound the contribution of these points below by $\delta^{s-\kappa}$ and then adding up $\approx t\delta^{-1}$ such intervals (this is where Lemma~\ref{transversality} is needed to ensure disjointness).

If $t \lessapprox \delta^{\alpha}$ this proves the lower bound of \eqref{maininequality}. If $t \gtrapprox \delta^{\alpha}$, then to adjust this argument to accommodate the requirement that $(x,x_1,x_2,x_3) \in \Lambda$, group the intervals $I_k$ of length $\delta/t$ into larger intervals $J_j$ of length $\delta^{\alpha}/t$, so that each group contributes $\approx \delta^{\alpha-1} \delta^{s-\kappa}$ to the lower bound. It suffices to show that for fixed $x,x_1, x_3$ with 
\[ |z-z_1|, |z-z_3| \geq t/2, \]
and fixed $j$, the set 
\begin{multline*} E:= \{ x_2 = (z_2, t_2) \in \mathbb{H} : d_E(z_2, \ell(z_1, z_3) ) < \delta^{\alpha}, \\
\text{ \eqref{equivrelation} holds for some $\theta \in H_2(x) \cap J_j$} \}  \end{multline*}
is contained in a Euclidean ball of radius $\approx \delta^{\alpha}$; the excision of this set will therefore not harm the lower bound of $\delta^{\alpha-1+s-\kappa}$ since $\delta^{\alpha s}$ is much smaller than $\delta^{\alpha-1+s-\kappa}$, by the definition of $\alpha$ in \eqref{alphadefn} (provided the $O(\eta)$ factor is chosen sufficiently large). 

To see that $E$ is contained in a ball of radius $\approx \delta^{\alpha}$, fix some $x_2 = (z_2, \tau_2) \in E$. The projection of $E$ down to $\mathbb{R}^2 \times \{0\}$ will be shown to be contained in 
\begin{equation} \label{twolines} \mathcal{N}_{\delta^{\alpha}}( \ell(z_1, z_3) ) \cap  \mathcal{N}_{C\delta^{\alpha}}(\ell(z,z_2)), \end{equation}
where $\mathcal{N}$ refers to Euclidean neighbourhood. The first set in the intersection comes from the definition of $E$. For the second set, by \eqref{equivrelation} the line $\ell(z,z_2)$ is at an angle of $\theta$ to the $x$-axis (up to an error $\lesssim \delta/t$), where $\theta$ is the angle from \eqref{equivrelation}. Since by definition of $E$ the corresponding angles of all other points in $E$ have been grouped into one interval of length $\delta^{\alpha}/t$, all other lines $\ell(z,z_2')$ with $x_2' \in E$ are within an angle $\lesssim \delta^{\alpha}/t$ of the line $\ell(z,z_2)$. Since by \eqref{equivrelation} all points $z_2' \in E$ satisfy $|z_2'-z| \leq 2t$, it follows that the part of $E$ in all of these lines is contained in $\mathcal{N}_{C\delta^{\alpha}}(\ell(z,z_2))$ for some large enough constant $C$. This proves the projection of $E$ down to $\mathbb{R}^2 \times \{0\}$ is contained in the set in \eqref{twolines}. The set in \eqref{twolines} is contained in a ball of radius $\approx \delta^{\alpha}$; this follows from $\delta^{O(\eta)}$-transversality of the lines $\ell(z_1,z_3)$ and $\ell(z,z_2)$. This transversality is a simple geometric consequence of the angle separation assumption on the sets $H_i(x)$; an explicit proof is given in \cite{THarris2}. It remains to bound the distances between the last coordinate. By \eqref{equivrelation} and the preceding argument, any two points $(z_2,\tau_2)$ and $(z_2', \tau_2')$ in $E$ satisfy 
\begin{equation} \label{three} |z_2-z_2'| \lessapprox \delta^{\alpha}, \quad \left\lvert \pi_{V_{\theta}^{\perp}}(z-z_2) \right\rvert < \delta, \quad  \left\lvert \pi_{V_{\theta'}^{\perp}}(z-z_2') \right\rvert < \delta, \end{equation}
\begin{equation} \label{tcomp} \left\lvert \tau -\tau_2 - \frac{1}{2} \omega \left( \pi_{V_{\theta}}(z), \pi_{V_{\theta}^{\perp}}(z)\right)+ \frac{1}{2} \omega \left( \pi_{V_{\theta}}(z_2), \pi_{V_{\theta}^{\perp}}(z_2)\right)  \right\rvert < \delta, \end{equation}
and 
\begin{equation} \label{tcomp2} \left\lvert \tau-\tau_2' - \frac{1}{2} \omega \left( \pi_{V_{\theta'}}(z), \pi_{V_{\theta'}^{\perp}}(z)\right)+ \frac{1}{2} \omega \left( \pi_{V_{\theta'}}(z_2'), \pi_{V_{\theta'}^{\perp}}(z_2')\right)  \right\rvert < \delta, \end{equation}
for some $\theta$ and $\theta'$. Combining the second and third parts of \eqref{three} with \eqref{tcomp} and \eqref{tcomp2} respectively yields 
\begin{equation} \label{trick} \left\lvert \tau -\tau_2 - \frac{1}{2} \omega(z,z_2)  \right\rvert \lesssim \delta, \quad \left\lvert \tau -\tau_2' - \frac{1}{2} \omega(z,z_2')  \right\rvert \lesssim \delta. \end{equation}
Combining this with the first part of \eqref{three} and using the triangle inequality gives 
\[ \left\lvert \tau_2 - \tau_2' \right\rvert \lessapprox \delta^{\alpha}, \]
which proves that $E$ is contained in a Euclidean ball of radius $\approx \delta^{\alpha}$, and finishes the proof of the lower bound of \eqref{maininequality}.

For the upper bound, the case $t \lessapprox \delta^{\alpha}$ follows by integrating over $(x,x_1,x_2,x_3)$ and using the Frostman condition on $\nu$, so assume that $t \gtrapprox \delta^{\alpha}$. Let
\[ A = A(x_1,x_2,x_3) = \{ x \in Z' : (x,x_1,x_2,x_3) \in \Lambda \text{ and } x \sim_i x_i \text{ for all } i \}. \]
The upper bound in \eqref{maininequality} will be shown by bounding $\nu(A)$ and then integrating over $(x_1,x_2,x_3)$. Let 
\[ A' = \{ x  \in A: |\tau-\tau_i|/10 \leq |z-z_i| \text{ for all } i \}, \]
where $x= (z,\tau)$. By similar working to that used to show \eqref{trick}, 
\[ A' \subseteq G^{-1}(B_E(0, C \delta)), \]
for some large constant $C$, where $G: \mathbb{R}^3 \to \mathbb{R}^3$ is the affine map
\[ G(z,\tau) = \begin{pmatrix} \tau-\tau_1 - \frac{1}{2} \omega( z, z_1 ) \\
\tau-\tau_2 - \frac{1}{2} \omega( z, z_2 ) \\
\tau-\tau_3 - \frac{1}{2} \omega( z, z_3 ) \end{pmatrix}. \]
As in the left projection case (\cite{THarris2}), the Jacobian satisfies 
\[ \left\lvert \det DG \right\rvert = \frac{1}{4} \left\lvert \omega(z_1, z_2)  +  \omega(z_2, z_3) +  \omega(z_3, z_1) \right\rvert \gtrapprox t \delta^{\alpha}, \]
by the definition of $\Lambda$. Hence 
\[ A' \subseteq G^{-1}(B_E(0, C \delta))  \subseteq B_E( G^{-1}(0), t^{-1}\delta^{1-\alpha-O(\eta)} ). \]
It follows that 
\begin{equation} \label{detbound} \nu(A') \lessapprox \delta^{(1-\alpha)s} t^{-s}. \end{equation}
To bound $\nu(A)$ it remains to bound $\nu(A \setminus A')$. If $x \in A \setminus A'$ then $|z-z_i| < |t-t_i|/10$ for some $i$, so by the condition $x \sim_i x_i$ and by similar working to \eqref{cauchyschwarz}, there exists $\theta$ such that
\begin{align*} d(x, x_i) &\lesssim |t-t_i| \\
&\lesssim \left\lvert t-t_i - \frac{1}{2} \omega \left( \pi_{V_{\theta}}(z), \pi_{V_{\theta}^{\perp}}(z)\right)+ \frac{1}{2} \omega \left( \pi_{V_{\theta}}(z_i), \pi_{V_{\theta}^{\perp}}(z_i)\right)  \right\rvert \\
&< \delta. \end{align*} 
Hence $\nu(A \setminus A') \lesssim \delta^{s}$. Combining with \eqref{detbound} gives
\[ \nu(A) \lessapprox \max\left\{ \delta^{(1-\alpha)s} t^{-s}, \delta^{s} \right\} \lessapprox \delta^{(1-\alpha)s} t^{-s}, \]
since $t \lesssim 1$. Integrating over $x_1, x_2, x_3$ gives
\[ \nu^4\left\{ (x,x_1,x_2,x_3) \in \Lambda : x \sim_i x_i \text{ for all } i \right\} \lessapprox  \delta^{(1-\alpha)s} t^{s}, \]
which is the upper bound of \eqref{maininequality}. 

If $t \lessapprox \delta^{\alpha}$, then combining the lower and upper bounds of \eqref{maininequality} gives 
\[ \nu(Z) t^3 \delta^{3(s-\kappa-1)} \lessapprox t^{3s}. \]
Since $s >1$, this simplifies to 
\[ \nu(Z) \lessapprox \delta^{3\alpha(s-1) - 3(s-\kappa-1)}, \]
and therefore $\nu(Z) \leq \delta^{\eta}$ by the definition of $\alpha$ in \eqref{alphadefn}. This finishes the proof if $t \lessapprox \delta^{\alpha}$. 

Now assume $t \gtrapprox \delta^{\alpha}$. In this case the lower and upper bounds of \eqref{maininequality} give 
\begin{equation} \label{nuzbound}  \nu(Z) t^3 \delta^{3(s-\kappa-1)} \lessapprox \delta^{(1-\alpha)s} t^{s}. \end{equation}

Since $s>1$, using $t \gtrapprox \delta^{\alpha}$ and simplifying gives
\begin{align*} \nu(Z) &\lessapprox \delta^{ (1-\alpha)s + \alpha(s-3) -3(s-1-\kappa) } \\
&\approx \delta^{ (1-\alpha)s  + \alpha(s-3)- 3(s-1)\alpha } &&\text{(by \eqref{alphadefn})} \\
&= \delta^{s(1-3\alpha)} \\
&= \delta^{\frac{3s}{s-1} \left( \kappa - \frac{2(s-1)}{3} - O(\eta)\right) } &&\text{(by \eqref{alphadefn})}. \end{align*}
%
%
Hence
\[ \nu(Z) \leq \delta^{\eta}, \]
by the assumption $\eta \ll  \kappa - \frac{2(s-1)}{3} $ in \eqref{etaassumption}.
%
This proves the lemma. \end{proof}

\begin{corollary} \label{1dbound} Let $A \subseteq \mathbb{H}$ be a Borel set. If $\dim_E A > 1$, then  
\[ \dim_E P_{\mathbb{V}_{\theta}^{\perp}} (A) \geq  
\frac{2+\dim_E A}{3}\]
for a.e.~$\theta \in [0,\pi)$, and if $\dim_{\mathbb{H}} A>2$, then 
\[ \dim_{\mathbb{V}_{\theta} \backslash \mathbb{H}} P_{\mathbb{V}_{\theta}^{\perp}}^R (A) \geq 
\frac{1+\dim_{\mathbb{H}} A}{3}, \]
for a.e.~$\theta \in [0,\pi)$.  \end{corollary}
\begin{proof} The Euclidean part follows from \cite[Lemma~2.1]{THarris2}, which says that any result of the type in Lemma~\ref{paradox} implies a corresponding projection theorem with lower bound $s-\kappa$ for sets of dimension $s$. 

The non-Euclidean part for right coset projections follows from the Euclidean bound, the dimension comparison principle and Lemma~\ref{locallipschitz}. \end{proof}

\section{Open questions}
\subsection*{Sharp Euclidean lower bounds} The Euclidean lower bound in Theorem~\ref{Hnbounda} is probably not sharp in the entire range. So the first obvious way to further this work would be to improve this bound, ideally finding sharp dimension distortion bounds. Since the projection maps are now viewed as maps from $\R^{2n+1}$ to $\R^{2n-m+1}$, purely Euclidean methods could in principle be applied to improve dimension distortion bounds. For instance, Fourier restriction methods used for example in \cite{OberOber} might lead to improvements. As we showed, when studying the problem as a Euclidean one, left and right coset projections cause the same dimension distortion. Therefore improving the bound in this direction could further improve the bound for the two problems relative to the more natural metrics that go with each one. 

\subsection*{Sharp \texorpdfstring{$\VH$}{} lower bounds} The method we employed here was to study the problem as a Euclidean one, and then apply the Dimension Comparison Principle to obtain dimension distortion bounds with respect to the more natural metric $\dvh$. So our bounds are obtained considering the worst dimension distortion by projections and the worst dimension drop by dimension comparison. In principle, these two things need not happen simultaneously so better bounds could potentially be obtained by considering the maps $P_{\V^{\perp}}^R$ as maps from $(\Hn,d_{\Hn})$ to $(\V^{\perp},\dvh)$ or $(\V^{\perp},d_E)$ and estimating energy integrals with respect to these metrics directly.

\subsection*{Projections of subsets with specific structure}
In \cite{HProj1} the authors gave evidence to their conjectured almost sure lower bound, in $\He$, by exhibiting some subsets of $\He$ with specific structure that do adhere to their conjecture. For instance, if the set $S$ is either a $\mathcal{C}^1$ curve or a $\mathcal{C}^1 $ surface then $\dim_{\He}P_{\theta}^LS\geq\dim_{\He}S$ for all but at most 2 values of $\theta$. Does something similar hold in higher dimensions for the projections $P_{\V^{\perp}}^L$ and/or $P_{\V^{\perp}}^R$? 

\subsection*{Structure of \texorpdfstring{$(\V^{\perp},\dvh)$}{vertical subgroups with the quotient metric}}
This problem was mentioned to in Section \ref{rightcoset}. The properties of this space discussed in that section hint that it might have the structure of a non-equiregular Carnot-Carath\'{e}odory space. So the problem is that of finding bracket generating vector fields in $\R^{2n-m+1}$ such that $\R^{2n-m+1}$ with the induced Carnot-Carath\'{e}odory distance is isometrically (or at least bi-Lipschitz) equivalent to $(\Vp,d_{\VH})$. Such a description of the space may also lead to improvements in dimension distortion bounds by projections as it could provide a better understanding of the metric itself.


\appendix
\section{A slicing result} \label{Appendix}
Let $\mathcal{H}^m$ denote the $m$-dimensional Hausdorff measure on Euclidean space, with respect to the Euclidean metric. Let $\mathcal{M}(A)$ be the class of compactly supported, nonzero, finite Radon measures on a set $A \subseteq \mathbb{H}^n$. Let $N(E, \delta)$ be the $\delta$-neighbourhood of a set $E \subseteq \mathbb{H}^n$ with respect to the Euclidean metric. 
\begin{theorem}[A slicing result]\label{EuclideanSlice}
Let $A\subset\Hn$ be a Borel set such that $\dim_A>m+1$ and $0<\Ha^{\dim_E A}A<\infty$ for $1 \leq m \leq n$. Then for $\munm$-almost every $V \in G_h(n,m)$,
\[\Ha^m(\{v \in\V:\ \dim_E[A\cap(\Vp v)]=\dim_E A-m\})>0\]
\end{theorem}
Note the great deal of similarity between this theorem and Theorem~1.5 in \cite{Hproj2}. In fact the proof follows the same techniques and only differ in that here we consider Hausdorff dimension with respect to the Euclidean metric.

\begin{proof}
Eilenberg's inequality (Theorem~13.3.1 in \cite{GeoIneq}) tells us that for every $\V$,
\[\dim_E[A\cap(\Vp v)]\leq\dim_E A-m,\  \text{for $\Ha^m$- almost every }v\in\V.\]
Therefore, we only need to prove the dimension lower bound. For this, we will make use of sliced measures in the sense of \cite{Mattila2} (Section 10.1). By Eq. (10.6) in \cite{Mattila2}, we know that for $\mu\in\mathcal{M}(A)$ there exists a family of measures $\mu_{\Vp v}$, defined for $\Ha^m$-a.e. $v\in\V$, each supported on $\Vp v$, such that for any non-negative continuous function $\varphi$ compactly supported on $\mathbb{H}^n$ and any Borel set $B\subset \V$, the map $v \mapsto \int_{\Vp v}\varphi(u) \, d\mu_{\Vp v}(u)$ is Borel measurable and satisfies
\begin{equation}\label{slicemeasure}
\int_B\int_{\Vp v}\varphi(u) \, d\mu_{\Vp v}(u) \, d\Ha^m(v) \leq \int_{P_{\V}^{-1}(B)}\varphi(u) \,  d\mu(u),
\end{equation}
with equality if $P_{\V\#}\mu\ll\Ha^m$. In particular if $P_{\V\#}\mu\ll\Ha^m$, $\int_{\V}\mu_{\Vp v}(A\cap \Vp v) \, d\Ha^m(v)=\mu(A)>0$, so that at least a $\Ha^m$-positive measure set of the measures $\mu_{\Vp v}$ are in $\mathcal{M}(A\cap \Vp v)$. Hence, we want to pick a measure $\mu\in\mathcal{M}(A)$ such that $P_{\V\#}\mu\ll\Ha^m$ for $\munm$- almost every $\V$. As the next claim will show, this is possible precisely when $\dim_E A>m+1$.

 \textbf{Claim}: Let $\sigma>m+1$ and assume $\mu\in\mathcal{M}(\Hn)$ satisfies $\mu(B_E(p,r))\leq r^\sigma$ for all $p\in\Hn$ and $r>0$. Then $P_{\V\#}\mu\ll\Ha^m\lfloor_{\V}$ for $\munm$- almost every $\V$.
 
 To see this, denote by $\pi:\Hn\to\Rnn$ the bundle map $\pi(z,t)=z$, and note that $P_{\V\#}\mu(B_{\V}(v,r))=\mu(P_{\V}^{-1}(B_{\V}(v,r)))=\mu(\{p\in\Hn:|P_{\V}(p)-v|<  r\})$. Now, Theorem~2.12 in \cite{Mattila1} tells us that $P_{\V\#}\mu\ll\Ha^m$ if and only if 
\[ \liminf_{\delta\to0}\delta^{-m}P_{\V\#}\mu(B_{\V}(v,\delta))<\infty \quad \text{for } P_{\V\#}\mu\text{-almost every } v\in\V. \]
Using Fatou's lemma, and Fubini (see e.g.~Theorem~1.14 in \cite{Mattila1}), we compute:
 \begin{align*} &\int_{\Gh}\int_{\V} \liminf_{\delta\to0}\delta^{-m}P_{\V\#}\mu(B_{\V}(v,\delta))\, dP_{\V\#}\mu(v) \, d\munm(V)\\
 &\quad \leq  \liminf_{\delta\to0}\delta^{-m}\int_{\Gh}\int_{\V}P_{\V\#}\mu(B_{\V}(v,\delta))\, dP_{\V\#}\mu(v) \, d\munm(V)\\
 &\quad =\liminf_{\delta\to0}\delta^{-m}\int_{\Hn}\int_{\Hn}\munm\left\{V\in\Gh: |P_{\V}(p)-P_{\V}(q)|<\delta \right\}\, d\mu(q)\, d\mu(p)\\
 &\quad \leq\liminf_{\delta\to0}\delta^{-m}\int_{\Hn}\int_{\Hn}\munm\left\{V\in\Gh: |\pi_{\mathbb{V}}(\pi(p))-\pi_{\mathbb{V}}(\pi(q)) |<\delta \right\} \, d\mu(q) \, d\mu(p)\\
 &\quad \lesssim\int_{\Hn}\int_{\Hn}|\pi(q)-\pi(p)|^{-m} \, d\mu(q) \, d\mu(p), \end{align*}
where the last step follows from Lemma~2.4 in \cite{Hproj2}. We now focus our attention on showing finiteness of this last integral. Since $\supp(\mu)$ is compact, we can fix $R>0$ such that $\supp(\mu)\subset B_E(0,R)$. For $z\in\Rnn$ the set $\{q\in\Hn: |\pi(q)-z|\leq r\}$ is a cylinder with radius $r$, so $\{q\in\Hn: |\pi(q)-z|\leq r\}\cap \supp(\mu)\subset B_E^{2n}(z,r)\times[-R,R]$. This cylinder can be covered by at most $\left\lceil Cr^{-1} \right\rceil$ balls of radius $r$, where $C=C(n,R)$ is independent of $z$ and $r$. It follows that $\mu(\{q\in\Hn: |\pi(q)-z|\leq r\})\lesssim r^{\sigma-1}$. Therefore,
\begin{align*}
\int_{\Hn}|\pi(q)-z|^{-m}d\mu(q) &=\int_0^{\infty}\mu\left(\left\{q\in\Hn: |\pi(q)-z|\leq r^{-1/m}\right\}\right)\, dr\\
&=\int_0^{1}\mu\left(\left\{q\in\Hn: |\pi(q)-z|\leq r^{-1/m}\right\}\right)\, dr\\
&\quad +\int_1^{\infty}\mu\left(\left\{q\in\Hn: |\pi(q)-z|\leq r^{-1/m}\right\}\right)\, dr\\
&\lesssim \mu(\Hn)+\int_1^{\infty}r^{\frac{1-\sigma}{m}}\, dr.
\end{align*}
Since $\sigma-1>m$ it follows that $\int_1^{\infty}r^{\frac{1-\sigma}{m}}\, dr<\infty$. This tells us that 
\[\int_{\Hn}\int_{\Hn}|\pi(q)-\pi(p)|^{-m} \, d\mu(q) \, d\mu(p)\lesssim \mu(\Hn)\left(\mu(\Hn)+\int_1^{\infty}r^{\frac{1-\sigma}{m}} \, dr\right)<\infty,\]
 which proves the claim.
   
By Frostman's lemma, if $\dim A=\alpha>m+1$ with $0<\Ha^{\alpha}(A)<\infty$, then we may choose $\mu\in\mathcal{M}(A)$ to be a suitable restriction of $\Ha^{\alpha}$ such that $\mu(B_E(p,r))\leq r^{\alpha}$ for all $p\in\Hn$ and $r>0$. From the claim, we know $P_{\V\#}\mu\ll\Ha^m$ for $\munm$-almost every $V \in G_h(n,m)$. As noted before, it follows that for $\munm$-almost every $V \in G_h(n,m)$, the measure $\mu_{\Vp v}$ is in $\mathcal{M}(A\cap \Vp v )$ for all $v$ in a set of positive $\Ha^m$ measure.

We now aim to show that if $m+1<s<\alpha$, then for $\munm$-almost every $V \in G_h(n,m)$, 
\begin{equation}\label{sigmaenergy}
I_{s-m}(\mu_{\Vp v},d_E)<\infty\text{ for }\Ha^m\text{-a.e.~$v\in\V$.} 
\end{equation}
By Fatou's lemma, Tonelli's theorem, and by applying \eqref{slicemeasure} with $B=B(v, \delta)$ and letting $\delta \to 0$, we can compute:
\begin{align*}
&\int\int_{\V}I_{s-m}(\mu_{\Vp v},d_E)\, d\Ha^m(v)\, d\munm(V)\\
&\quad \leq \liminf_{\delta\to0} \delta^{-m}\int\int_{\V}\int_{\Vp v}\int_{N(\Vp v,\delta)}|p-q|^{m-s} \, d\mu(p) \, d\mu_{\Vp v}(q) \, d\Ha^m(v)\, d\munm(V)\\
&\quad \leq \liminf_{\delta\to0} \delta^{-m}\int\int_{\V}\int_{N(\Vp v,\delta)}\int_{\Vp v}|p-q|^{m-s}\, d\mu_{\Vp v}(q)\, d\mu(p)\, d\Ha^m(v) \, d\munm(V)\\
&\quad \leq \liminf_{\delta\to0} \delta^{-m} \\
&\qquad \times \int \int_{\Hn}\int_{\{v\in\V:d_E(p,\Vp v)\leq\delta\}}\int_{\Vp v}|p-q|^{m-s}\, d\mu_{\Vp v}(q) \, d\Ha^m(v)\, d\mu(p) \, d\munm(V).
\end{align*}
 Now we apply \eqref{slicemeasure} to the inner double integral, use Tonelli's theorem, and apply Lemma~2.4 from \cite{Hproj2} to get
 \begin{align*}
 &\int\int_{\V}I_{s-m}(\mu_{\Vp v},d_E)\, d\Ha^m(v) \, d\munm(V)\\
 &\quad \leq \liminf_{\delta\to0}\delta^{-m} \int \int_{\Hn}\int_{\{q:|P_{\V}(p-q)|\leq \delta\}}|p-q|^{m-s} \, d\mu(q) \, d\mu(p) \, d\munm(V)\\
 &\quad \lesssim \int_{\Hn}\int_{\Hn}|p-q|^{m-s}|\pi(p)-\pi(q)|^{-m} \, d\mu(q) \, d\mu(p).
 \end{align*}
 This last integral is not quite $I_s(\mu,d_E)$, and in fact the singularity in the kernel $|q|^{m-s}|\pi(q)|^{-m}$ is stronger than the one in the kernel $|q|^{-s}$. Nevertheless, we will show this integral is finite following the same approach as in the proof of Theorem~1.5 in \cite{Hproj2}, by showing that the inner integral is finite for all $p$ and using the fact that $\mu(\Hn)<\infty$.
 
 If we denote by $L_{-p}$ the Euclidean left translation by $-p$, the inner integral can be written as 
 \[\int_{\Hn}|p-q|^{m-s}|\pi(p)-\pi(q)|^{-m} \, d\mu(q)=\int_{\Hn}|q|^{m-s}|\pi(q)|^{-m}\, dL_{-p\#}\mu(q).\]
 Moreover, it is clear that $L_{-p\#}\mu(\Hn)=\mu(\Hn)$ and $L_{-p\#}\mu(B_E(q,r))\leq r^s$ for every $q \in\Hn$ and $r>0$. Furthermore, since $\mu$ is compactly supported, by scaling we may assume the support of $L_{-p\#}\mu$ is contained in $B_E(0,1)$. Therefore it is enough to show that \[\int_{\Hn}|q|^{m-s}|\pi(q)|^{-m}\, d\mu(q) \lesssim 1,\]
 whenever $\mu\in\mathcal{M}(B_E(0,1))$ and satisfies $\mu(B_E(p,r))\leq r^{\alpha}$ for all $p\in\Hn$ and $r>0$. Writing $q=(\zeta,\tau)$ we have 
 \begin{align*}
 \int_{\Hn} |q|^{m-s}|\pi(q)|^{-m} \, d\mu(q) &=\int \left\lvert|\zeta|^2+\tau^2\right\rvert^{\frac{m-s}{2}}|\zeta|^{-m} \, d\mu\\
 &\sim \int_{\{|\zeta|\geq|\tau|\}}|q|^{-s}\, d\mu +\int_{\{|\zeta|<|\tau|\}}|\tau|^{m-s}|\zeta|^{-m}\, d\mu\\
 &=:\mathcal{I}_1+\mathcal{I}_2. \end{align*}
 We look at these two quantities separately, the first one being the easier to bound. Indeed, using a change of variables and recalling our choice of $s$,
 \begin{align*}
 \mathcal{I}_1\leq \int_{\Hn}|q|^{-s} \, d\mu&=\int_0^{\infty}\mu\left(\left\{q:|q|\leq r^{-1/s}\right\} \right)\, dr\\
 &=\int_0^{\infty}\mu\left(B_E\left(0,r^{-1/s} \right)\right) \, dr\\
 &=s\int_0^{\infty}\mu\left(B_E(0,u) \right)u^{-s-1} \, du\\
 &\leq s\int_0^{1}u^{\alpha-s-1} \, du+ s\mu(\Hn)\int_{1}^{\infty}u^{-s-1} \, du<\infty.
 \end{align*}
 To bound the second integral we first split the domain of integration: 
\begin{multline*} \left\{(\zeta,\tau)\in B_E(0,1): |\zeta|<|\tau| \right\} \\
=\bigcup_{i=0}^{\infty} \left\{(\zeta,\tau)\in B_E(0,1):2^{-i-1}|\tau|\leq|\zeta| < 2^{-i}|\tau|\right\}=:\bigcup_{i=0}^{\infty} A_i. \end{multline*} 
 Note that for $(\zeta,\tau)\in A_i$, $|\zeta|^{-1}\sim 2^{i}|\tau|^{-1}$. Therefore,
 \begin{align*}
 \mathcal{I}_2&\sim\sum_{i=0}^{\infty}\int_{A_i}(2^{-i}|\tau|)^{-m}|\tau|^{m-s} \, d\mu\\
 &=\sum_{i=0}^{\infty}\int_{A_i}2^{im}|\tau|^{-s} \, d\mu \\
 &\sim\sum_{i,j =0}^{\infty} \int_{A_{i,j}}2^{im}(2^{-j})^{-s} \, d\mu=\sum_{i,j}2^{im+js}\mu(A_{i,j}),
 \end{align*}
 where \[A_{i,j}=\left\{(\zeta,\tau)\in B_E(0,1):2^{-i-j-2}\leq|\zeta|\leq2^{-i-j},\ 2^{-j-1}\leq|\tau| < 2^{-j}\right\}.\]
 To estimate $\mu(A_{i,j})$ we see that  $A_{i,j}\subset B_E^{2n}(0,2^{-i-j})\times[-2^{-j}, 2^{-j}]$. Hence, there exists a constant $C>0$ independent of $i$ and $j$ such that $A_{i,j}$ can be covered by at most $C\frac{2^{-j}}{2^{-i-j}}=C2^i$ balls of radius $2^{-i-j}$. The Frostman condition on $\mu$ now tells us that $\mu(A_{i,j})\lesssim 2^i 2^{-\alpha(i+j})$. Going back to the sum we are trying to bound, we get
 \[\sum_{i,j=0}^{\infty} 2^{im+js}\mu(A_{i,j})\lesssim\sum_{i,j=0}^{\infty}2^{i(m+1-\alpha)+j(s-\alpha)},\]
 which is finite since $m+1-\alpha$ and $s-\alpha$ are both negative.

 Now to complete the proof of the proposition, for $V\in\Gh$ write
 \[E_{\V}:=\{v\in\V:\mu_{\Vp v}(\Hn)>0\},\]
 so that for $v\in E_{\V}$, $\mu_{\Vp v}\in\mathcal{M}(A\cap(\Vp v))$. Since, by the claim, we know  $P_{\V\#}\mu\ll\Ha^m$, equality in \eqref{slicemeasure} with $B=\V$ tells us that $\Ha^m(E_{\V})>0$. Furthermore, by the previous computation it follows that if $m+1<s<\alpha$ then for $\munm$-almost every $V\in\Gh$, the energy $I_{s-m}(\mu_{\Vp v},d_E)$ is finite for $\Ha^m$-almost every $v\in E_{\V}$. This tell us that $\dim_E[A\cap(\Vp v)]\geq s-m$. Since $E_{\V}$ is independent of $s$ and $\alpha$, the theorem follows by letting $s\to\alpha$.
 \end{proof}

\section{Construction of a product set with prescribed Euclidean and Heisenberg dimension} \label{appendix2}
In this section we outline the construction of the set in \eqref{productconditions} required in part of the proof of Theorem~\ref{Hnbounda}, specifically for the sharpness of the lower bound in \eqref{koranyi2}. Given $\alpha \in [0, 2n-1]$, we require a compact set of the form $A = \mathcal{C}_{\alpha} \times I$ such that 
\begin{equation} \label{dimreq} \dim_E A = \alpha+1, \quad \dim_{\mathbb{H}^n} A = \alpha+2, \end{equation}
where $I \subseteq \mathbb{R}$ is a compact interval. There are two cases; either 
\begin{equation} \label{evencase} \alpha = 2j + \beta, \quad j \in \{0, 1, \dotsc, n-1\}, \quad 0 \leq \beta \leq 1, \end{equation}
or 
\begin{equation} \label{oddcase} \alpha = 2j+1 + \beta, \quad j \in \{0, 1, \dotsc, n-2 \}, \quad 0 < \beta < 1. \end{equation}
In the first case, let $\mathcal{C}_{\beta}$ be a Cantor set in $\mathbb{R}$ with finite, nonzero $\beta$-dimensional Euclidean Hausdorff measure, and let 
\[ A'  = \left(\mathcal{C}_{\beta} \times \{0 \} \right) \times \mathbb{C}^j \times \{0\}^{2(n-1)-2j} \times \mathbb{R}, \]
which clearly satisfies $\dim_E A' = \alpha+1$. Using $\mathcal{H}^a_b$ to denote the $a$-dimensional Euclidean Hausdorff measure living on $\mathbb{R}^b$, the measure
\[ \mu := \left(\mathcal{H}^{\beta}_2  \times \mathcal{H}^{2j}_{2(n-1)} \times \mathcal{H}^1_1 \right)|_{A'}, \]
is nonzero and supported on $A'$, and satisfies 
\[ \mu\left(B_{\mathbb{H}^n}((z,t), r ) \right) \lesssim r^{\alpha+2}, \]
for all $(z,t) \in A'$ and $0< r < 1$. This can be proved similarly to the proof that the Hausdorff dimension of $\mathbb{H}^n$ is $2n+2$; using invariance by left translation and that the coordinates from $\mathcal{C}_{\beta} \times \{0 \}$ vanish in the symplectic form. Hence $\dim_{\mathbb{H}^n}(A') \geq \alpha+2$, and therefore $\dim_{\mathbb{H}^n}(A') = \alpha+2$ by dimension comparison. By using homogeneous dilations, this implies that the set
\[ A : = \left(\mathcal{C}_{\beta} \times \{0 \} \right) \times [0,1]^{2j} \times \{0\}^{2(n-1)-2j} \times [0, 1]. \]
satisfies \eqref{dimreq}. 
This finishes the construction in the case of \eqref{evencase}. The odd case in \eqref{oddcase} is similar, except that $A$ is defined by
\[  A= \left(\mathcal{C}_{\frac{1+\beta}{2}} \times \{0 \} \right)^2 \times [0,1]^{2j} \times \{0\}^{2(n-2)-2j} \times [0,1]. \]


\bibliographystyle{plain}

\bibliography{BiblioProj}

\end{document}